\documentclass[a4paper,12pt, leqno ]{article}
\usepackage{subcaption}
\usepackage{latexsym}
\usepackage{amsmath}
\usepackage{amssymb}
\usepackage{enumerate}
\usepackage{pict2e}
\usepackage{tikz}
\usepackage{theorem}

\usepackage{epic,eepic}
\usepackage{color}
\usepackage{float}

\usepackage{xcolor}

\newtheorem{theorem}{Theorem}[section]
\newtheorem{proposition}[theorem]{Proposition}
\newtheorem{lemma}[theorem]{Lemma}

\theorembodyfont{\rmfamily}
\newtheorem{proof}{\textmd{\textit{Proof.}}}

\newtheorem{remark}[theorem]{Remark}

\newtheorem{definition}[theorem]{Definition}

\newtheorem{acknowledgement}{\textmd{\textit{Acknowledgements.}}}

\makeatletter

\@addtoreset{equation}{section}
\makeatother

\newcommand{\qedd}{\hfill \Box}
\newcommand{\ve}{\varepsilon}

\newcommand{\R}{\ensuremath{\mathbb{R}}}

\newcommand{\Sph}{\ensuremath{\mathbb{S}}}

\usepackage{listings}
\lstdefinelanguage{Sage}[]{Python}
{morekeywords={False,sage,True},sensitive=true}
\definecolor{dblackcolor}{rgb}{0.0,0.0,0.0}
\definecolor{dbluecolor}{rgb}{0.01,0.02,0.7}
\definecolor{dgreencolor}{rgb}{0.2,0.4,0.0}
\definecolor{dgraycolor}{rgb}{0.30,0.3,0.30}

\title{The geometry of a positively curved Zoll surface of revolution}

\author{By  K. Kiyohara, S. V. Sabau
	\footnote{Corresponding author}
	, K. Shibuya}
\date{}
\begin{document}

\maketitle



\begin{abstract}

In this paper we study the geometry of
the manifolds of geodesics of a Zoll surface of positive Gauss curvature, show how these metrics induce Finsler metrics of constant flag curvature and give some explicit constructions.
\end{abstract}


\section{Introduction}
The study of Riemannian manifolds all of whose geodesics are closed has a long history (see \cite{Be} for historical remarks). In special, Riemannian manifolds whose geodesics are simple closed curves of equal length, manifolds called today {\it Zoll surfaces}, have remarkable geometrical properties extensively studied by many experts (see for instance \cite{Be}, \cite{LBM}, \cite{K1}, \cite{MS1}). 
 
 A Zoll surface of revolution $(\Lambda,g)$ 
 is the surface with the local coordinates
 $
 (r,\theta)\in [0,\pi]\times \R\slash 2\pi \mathbb Z,
 $
 and the metric
 \begin{equation}\label{Zoll metric }
 g=[1+h(\cos r)]^2 dr\otimes dr+\sin^2\ r d\theta\otimes d\theta,
 \end{equation}
 where $h:[-1,1]\to (-1,1)$ is an smooth function such that
 \begin{enumerate}
 	\item $h(-x)=-h(x)$, for any $x\in [-1,1]$, i.e. it is an odd function,
 	\item $h(-1)=h(1)=0$.
 \end{enumerate}
 
 This is a smooth Riemannian metric on $\Sph^2$ regarded as
 $
 (0,\pi)\times \R\slash 2\pi \mathbb Z\cup\{r=0\}\cup\{r=\pi\}.
 $
 
 It is known that, the general Zoll metrics near the standard one are parametrized 
 by odd functions $h:\Sph^2\to \R$, see \cite{Be} or \cite{LBM}.
 
 Amongst many other remarkable geometrical properties of Zoll surfaces we mention the fact that the manifold of oriented geodesics $M$ of a Zoll surface $(\Lambda=\Sph^2,g)$ is a smooth manifold diffeomorphic to $\Sph^2$. It is also known that, for any unit speed geodesic $\gamma$ of $(\Lambda,g)$,  the tangent space $T_{[\gamma]}M$ to $M$ in the point $[\gamma]$ is isomorphic to the space of normal Jacobi fields along $\gamma$ (see for instance \cite{Be}).  
 
 One fundamental question to ask is {\it what kind of natural geometrical structures are carried by the manifold of geodesics M of a Zoll surface $(\Lambda,g)$, and how are these related to the geometry of $(\Lambda,g)$?}
 
 Some answers are already known. For instance, it is known that $M$ can be endowed with a symplectic structure. 
  Moreover, a Riemannian metric is introduced by Besse (\cite{Be}, p. 62) on the manifold of geodesics. 
 
 In the present paper 
 \begin{itemize}
 	\item   we show that  {\it the manifold of geodesics of a positively curved Zoll metric naturally inherits a Finsler structure of constant flag curvature $K=1$.}
 	\item  Moreover, {\it we construct this Finsler structure and study its 
 	geometry.} Some examples are also given.
 \end{itemize}

The idea of existence of Finsler metrics of constant flag curvature on the manifold of geodesics of a Zoll metric is not new, originally belonging to Bryant (\cite{Br2002}). However, the concrete construction of such a Finsler metric, its geometrical properties or examples remain unknown until now. 

The Finsler metrics of positive constant flag curvature constructed in this paper are intimately related to the original Zoll metric in the sense that the geodesic foliation of the Zoll metric coincides with the indicatrix foliation of the Finsler metric, and the geodesic foliation of the Finsler metric coincides with the indicatrix (that is the unit sphere bundle) foliation of the Zoll metric. All geodesics of a $K=1$ Finsler metric constructed from a positively curved Zoll metric are closed and they intersect each other at a distance $\pi$ from the initial point.

The present paper clarifies the correspondence between a Finsler metric of positive constant flag curvature and a given positively curved Zoll metric. 
 To keep things simple, in the present paper, we restrict ourselves to the simplest case of determining Finsler metrics from Zoll surfaces of revolution, but the more general case of an arbitrary Zoll surface or the higher dimensional case can also be studied. We will consider some of these topics in a forthcoming research.

 \begin{center}
 	*
 \end{center}
  \quad Here is the structure of our paper. 
  
  In Section \ref{sec: Finsler surfaces} we recall basic facts about the geometry of Finsler surfaces and its indicatrix (the unit sphere bundle of a Finsler structure). Moreover, we present the peculiarities of Finsler metrics of constant flag curvature $K=1$ and recall 
   the existence theorem of constant sectional curvature Finsler structures due to Bryant (we refer the reader to the original papers \cite{Br1995}, \cite{Br2002}, or our presentation in \cite{SSS2012}). 
  
  In Section \ref{sec: Zoll surfaces} we recall basic properties of Zoll surfaces which provide the setting necessary to define the geometrical structures on the manifold of geodesics (our main reference here is \cite{Be}). 
  Moreover, in order to construct the manifold of geodesics, we explicitly compute the normal Jacobi fields along a geodesic of the Zoll metric, see  Proposition \ref{prop:Jacobi fields Y1 Y2}.
  
  Using all these, we move on to the construction of the manifold of geodesics $M$, in Section \ref{sec: manif of geod}, by giving the embedding $\iota:\Sigma\to TM$ of the indicatrix space in the tangent space. This is one of our main findings that will lead to the parametric equations of the Finsler indicatrix. Moreover, we introduce local coordinates on the manifold of geodesics in Subsection \ref{subsec:coordinates on M}.
  
  Now we are able to construct the Finsler metric in Section \ref{sec:the Finsler metric}, by giving the explicit form in coordinates of the parametric equations of the indicatrix curve, see Theorem \ref{thm: indicatrix param eqs}. We also prove that the positivity of the Gauss curvature of the original Zoll metric is in fact equivalent to the positive definiteness of the constant flag curvature Finsler metric in 
  Theorem \ref{thm: G vs indic curv}. Obviously, in the case $h=0$, the Zoll metric becomes the canonical constant Gauss curvature metric on the sphere and the induced Finsler metric is also the canonical constant Gauss curvature metric on the sphere (see Subsection \ref{Example: Riemannian case}). 
  
  In the final section we turn our attention to examples. The most ubiquitous case is when the function $h(x)$ is a polynomial in $x$. In this case the concrete form of the implicit equation of the Finsler indicatrix is given in  Theorem \ref{thm: implicit equation for h(x) polyn}. We apply this result to some concrete examples of positively curved Zoll surfaces constructed by ourselves, see Examples \ref{subsec: Eg. 1} and \ref{subsec: Eg. 2}. We also plot the indicatrix curves of the Finsler surfaces constructed (see Figures 	\ref{fig: indicatrices eg 1}). We point out that even though we did not write the explicit form of the corresponding Finsler fundamental function $F$, finding the implicit equations of the indicatrices is basically the same thing. Explicitly writing down these fundamental functionsis always possible, but since we have obtain the main geometrical properties of these Finsler metrics, this would not be of much use anyway. 
  
 \begin{acknowledgement}
 	We thank to H. Shimada and V. Matveev for many useful discussions.
 	\end{acknowledgement}

\section{Finsler surfaces}\label{sec: Finsler surfaces}

\subsection{The geometry of a Finsler surface}\label{subsec: Finsler surf}

A Finsler norm, or metric, on a real smooth, $n$-dimensional manifold
$M$ is a function $F:TM\to \left[0,\infty \right)$ that is positive and
smooth on $\widetilde{TM}=TM\backslash\{0\}$, has the {\it homogeneity property}
$F(x,\lambda v)=\lambda F(x,v)$, for all $\lambda > 0$ and all 
$v\in T_xM$, having also the {\it strong convexity} property that the
Hessian matrix
\begin{equation*}
g_{ij}=\frac{1}{2}\frac{\partial^2 F^2}{\partial y^i\partial y^j}
\end{equation*}
is positive definite at any point $u=(x^i,y^i)\in \widetilde{TM}$.

The fundamental function $F$ of a Finsler structure $(M,F)$ determines and it is determined by the (tangent) {\it indicatrix}, or the total space of the unit tangent bundle of $F$, namely
\begin{equation*}
\Sigma_F:=\{u\in TM:F(u)=1\}
\end{equation*}
which is a smooth hypersurface of $TM$. At each $x\in M$ we also have the {\it indicatrix at x}
\begin{equation*}
\Sigma_x:=\{v\in T_xM \ |\  F(x,v)=1\}=\Sigma_F\cap T_xM
\end{equation*}
which is a smooth, closed, strictly convex hypersurface in
$T_xM$. 

To give a Finsler structure $(M,F)$ is therefore equivalent to giving a smooth
hypersurface $\Sigma\subset TM$ for which the canonical projection
$\pi:\Sigma\to M$ is a surjective submersion and having the property
that for each $x\in M$, the $\pi$-fiber $\Sigma_x=\pi^{-1}(x)$ is
strictly convex including the origin $O_x\in T_xM$.

In order to study the differential geometry of the Finsler structure
$(M,F)$, it is convenient to consider the pull-back bundle
$\pi^*TM$ with the base manifold $\Sigma$ whose fibers over a point $u\in \Sigma$, $\pi(u)=x\in M$ are isomorphic 
to $T_xM$ 
(see
\cite{BCS2000}).

By defining an orthonormal moving coframing on $\pi^*TM$ with respect to the
Riemannian metric on $\Sigma$ induced by the Finslerian metric $F$, the
moving equations on this frame lead to the so-called Chern
connection. This is an almost metric compatible, torsion free connection of the
vector bundle $(\pi^*TM,\pi,\Sigma)$.

\quad We are going to restrict ourselves for the rest of the paper to the  two dimensional case. To be more precise, 
our manifold $\Sigma$ will be always 3-dimensional, and the manifold $M$ will be 2-dimensional, in the case it exists.

It is known (see for instance \cite{BCS2000}) that $\Sigma$ becomes a 3-dimensional Riemannian manifold with the metric
\begin{equation}
\omega^1\otimes \omega^1+\omega^2\otimes \omega^2+\omega^3\otimes \omega^3,
\end{equation}
that is $\{\omega^1,\omega^2,\omega^3\}$ is a $g$-orthonormal moving coframe on $\Sigma$. For later use we denote the natural dual basis by $\{\hat{e}_1,\hat{e}_2,\hat{e}_3\}$. 

It is also known that $\{\omega^1,\omega^2,\omega^3\}$ must satisfy the structure equations
\begin{equation}\label{Finsler structure eq}
\begin{split}
& d\omega^1= -I\omega^1\wedge\omega^3+\omega^2\wedge\omega^3\\
& d\omega^2= \omega^3\wedge\omega^1\\
& d\omega^3= K\omega^1\wedge\omega^2-J\omega^1\wedge\omega^3,
\end{split}
\end{equation}
where $I$, $J$, $K$ are smooth functions on $\Sigma$ called the invariants of the Finsler structure. More precisely, the functions $I$, $J$, $K$ are called the Cartan scalar, the Landsberg curvature and the flag curvature of $(M,F)$, respectively. Equivalently, we have
\begin{equation}
\begin{split}
& [\hat{e}_1,\hat{e}_2]=-K\hat{e}_3\\
& [\hat{e}_2,\hat{e}_3]=-\hat{e}_1\\
& [\hat{e}_3,\hat{e}_1]=-I\hat{e}_1-\hat{e}_2-J\hat{e}_3.
\end{split}
\end{equation}

A Finsler surface $(M,F)$ is Riemannian if and only if the Cartan scalar $I$ vanishes everywhere on $\Sigma$. 



A very usefull generalization of this notion is the {\it generalized Finsler structure} introduced 
by R. Bryant. In the two dimensional case a generalized Finsler structure is a 
coframing $\omega=(\omega^1,\omega^2,\omega^3)$ on a three dimensional manifold 
$\Sigma$ that satisfies some given structure equations (see \cite{Br1995}). 
By extension, one can study the generalized Finsler structure $(\Sigma,\omega)$ 
defined in this way ignoring even the existence of the underlying surface $M$. 

Observe  that in the case $n>2$, there will be no such 
globally defined coframing on the $2n-1$-dimensional manifold $\Sigma$. The reason 
is that even though the orthonormal frame bundle $\mathcal{F}$ over $M$ does admit 
a global coframing, it is a peculiarity of the $n=2$ dimensional case that  
$\mathcal{F}$ 
can be identified with $\Sigma$ (see also \cite{BCS2000}, p. 92-93 for concrete 
computations).

\begin{definition}
	 A 3-dimensional manifold $\Sigma$ endowed with a coframing $
	 \omega=(\omega^1,\omega^2,\omega^3)$ which satisfies the structure equations \eqref{Finsler structure eq}
	 will be therefore called a {\it generalized Finsler surface},
	 where $I$, $J$, $K$ are smooth functions on $\Sigma$, called the invariants of the generalized Finsler structure $
	 (\Sigma,\omega)$ (see \cite{Br1995} for details).
\end{definition}

As long as we work only with generalized Finsler surfaces, it might be possible that this generalized structure 
is not realizable as a classical Finslerian structure on a surface $M$. This imposes the following definition \cite{Br1995}.

\begin{definition}
	A generalized Finsler surface $(\Sigma,\omega)$ is said to be {\it amenable} if the leaf space $
	\mathcal{M}$ of the codimension 2 foliation defined by the equations $\omega^1=0$, $\omega^2=0$ is a smooth surface 
	such that the natural projection $\pi:\Sigma\to \mathcal{M}$   is a smooth submersion.
\end{definition}

 As R. Bryant emphasizes in  \cite{Br1995} the difference between a classical Finsler structure and a 
generalized one is global in nature, in the sense that {\it every generalized Finsler surface structure is locally 
	diffeomorphic to a classical Finsler surface structure. 
}\\
\quad The following fundamental result can be also found in \cite{Br1995}\\

\begin{theorem}
	 The necessary and sufficient condition for a generalized Finsler surface  $
	 (\Sigma,\omega)$ to be realizable as a classical Finsler structure on a surface are
	 \begin{enumerate} 
	 	\item the leaves of the foliation $\{\omega^1=0,\ \omega^2=0\}$ are compact;
	 	\item it is amenable, i.e. the space of leaves of the foliation 
	 	$\{\omega^1=0,\ \omega^2=0\}$ is a differentiable manifold $M$;
	 	\item the canonical immersion $\iota:\Sigma\to TM$, given by 
	 	$\iota(u)=\pi_{*,u}(\hat{e}_2)$, is one-to-one on each $\pi$-fiber $\Sigma_x$,
	 \end{enumerate}
	 where we denote by $(\hat{e}_1,\hat{e}_2, \hat{e}_3)$ the dual frame of the coframing $(\omega^1,\omega^2,\omega^3)$.
\end{theorem}

In the same source it is pointed out that if for example the $\{\omega^1=0,\ \omega^2=0\}$ leaves are not 
compact, or even in the case they are, if they are ramified, or if the curves  $\Sigma_x$ winds around origin in $T_xM$, 
in any of these cases, the generalized Finsler surface structure is not realizable as a classical Finsler surface.

 An illustrative example found in \cite{Br1995} is the case of an amenable generalized Finsler surface such 
that the invariant $I$ is constant, however $I$ is not zero. This kind of generalized structure is not realizable as a Finsler 
surface because $I\neq 0$ means that the leaves of the foliation $\{\omega^1=0, \ \omega^2=0\}$ are not compact. 
Indeed, in the case $I^2<4$, the  $\pi$-fibers $\Sigma_x$ are logarithmic spirals in $T_xM$.\\
\quad Let us return to the general theory of generalized Finsler structures on surfaces. 
By taking the exterior derivative of the structure equations (\ref{Finsler structure eq}) one obtains the {\it Bianchi equations 
	of the Finsler structure}:
\begin{equation*}
 J=I_2,\quad 
 K_3+KI+J_2=0,
\end{equation*}
where we denote by $I_i$ the directional derivatives with respect to the coframing $\omega$, i.e.
$df=f_1\omega^1+f_2\omega^2+f_3\omega^3,
$
for any smooth function $f$ on $\Sigma$.

Taking now one more exterior derivative of the last formula written above,  one obtains the Ricci identities 
with respect to the generalized Finsler structure
\begin{equation*}
\begin{split}
& f_{21}-f_{12}=-Kf_3\\
& f_{32}-f_{23}=-f_1\\
& f_{31}-f_{13}=If_1+f_2+Jf_3.
\end{split}
\end{equation*}
\quad{\bf Remarks.}
\begin{enumerate}
	\item  Remark first that the structure equations of a Riemannian surface are obtained from (\ref{Finsler structure eq}) by 
	putting $I=J=0$. 
	\item Since $J=I_2$, one can easily see that the necessary and sufficient condition for a generalized Finsler structure to 
	be non-Riemannian is $I\neq 0$.  
\end{enumerate}


\subsection{Bryant's existence Theorems}\label{subsec: Bryant theorems}

The existence of Finsler structures of constant flag curvature on the manifold of geodesics of a positively curved Zoll manifold was pointed out for the first time by R. Bryant (see \cite{Br2002}). We will recall in this section Bryant's results and reformulate them in a convenient form for our considerations in the following paragraphs. The theorems in this section are essentially equivalent to the results in  
\cite{Br2002}.

Observe that in the case $K=1$ the 
the structure equations \eqref{Finsler structure eq} can be written as 
\begin{equation}\label{K=1 Finsler structure eq 2}
\begin{split}
& d\omega^1= [-I\omega^1+\omega^2-J\omega^3]\wedge\omega^3\\
& d\omega^2= \omega^3\wedge\omega^1\\
& d\omega^3= \omega^1\wedge [-I\omega^1+\omega^2-J\omega^3],
\end{split}
\end{equation}
with the Bianchi equations
\begin{equation}\label{K=1 Finsler Bianchi}
J=I_2,\quad I+J_2=0,
\end{equation}
and that, by using the well-known formula 
$\mathcal{L}_X\omega=i_Xd\omega+d(i_X\omega)$, we have the following invariance formulas 
\begin{equation}\label{K=1 invariance}
\begin{split}
& \mathcal{L}_{\hat{e}_2}\omega^1=\omega^3,\quad 
\mathcal{L}_{\hat{e}_2}\omega^2=0,\quad
\mathcal{L}_{\hat{e}_2}\omega^3=-\omega^1,
\\
& \mathcal{L}_{\hat{e}_2}I=J,\quad \mathcal{L}_{\hat{e}_2}J=-I.
\end{split}
\end{equation}

We observe that in the case of a non-Riemannian Finsler surface with $K=1$, the remaining invariants $I$ and $J$ must be both non-vanishing smooth functions on $\Sigma$.

An elementary computation shows that 
\begin{equation}\label{invariance 2}
\begin{split}
& \mathcal{L}_{\hat{e}_2}[\omega^1\otimes \omega^1+\omega^3\otimes \omega^3]=0,\quad \mathcal{L}_{\hat{e}_2}[\omega^1\wedge \omega^3]=0,\\
& \mathcal{L}_{\hat{e}_2}[I\omega^1+J\omega^3]=0.\end{split}
\end{equation}




\begin{figure}[h]
	\begin{center}
		\setlength{\unitlength}{1cm}
		\begin{picture}(5,4)
		\put(0.3,3.9){\vector(1,0){2.5}}
		\put(0.2,3.5){\vector(1,-1){2.5}}
		\put(0,3.5){\vector(-1,-1){2.5}}
		\put(1.3,1.5) {$\pi$}
		\put(-1.5,1.5) {$\lambda$}
		\put(-0.2,3.8) {$\Sigma$}
		\put(3,3.8) {$TM$}
		\put(1.5,4){$\iota$}
		\put(-2.9,0.5) {$(\Lambda,g)  $}
		\put(2.5,0.5) {$(M,F) $}
		\end{picture}
		\caption{The Cartan double fibration.}
		\label{fig:Cartan fibration}
	\end{center}
\end{figure}
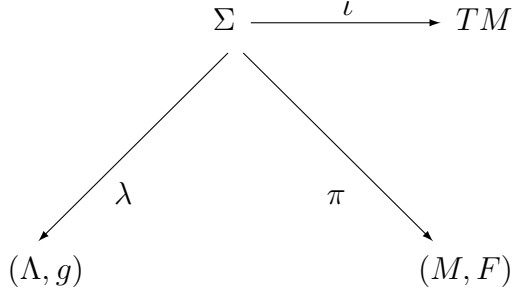

\bigskip

\bigskip




We recall that a Finsler structure $(M,F)$ is called {\it geodesically amenable} if the set 
$\Lambda:=\Sigma\slash_{  \langle \hat{e}_2\rangle}=\Sigma\slash_{\{\omega^1=0,\omega^3=0\}}$ of the integral curves of the vector field $\hat{e}_2$ can be given a structure of smooth manifold of dimension two such that the natural mapping $\lambda:\Sigma\to \Lambda$ is a smooth submersion (see Figure \ref{fig:Cartan fibration}). 

In general there is no natural metric on $\Lambda$, but in the case of $K=1$ such a metric does exist.

\begin{theorem}
	Let $(M,F)$ be a geodesically amenable Finsler surface with the canonical coframe $(\omega^1,\omega^2,\omega^3)$ and invariats $I,J$ and $K=1$ on the indicatrix bundle $\Sigma$. 
	
	Then, on the manifold of geodesics $\Lambda$ there exists a one-form $\varpi$ 
		\begin{equation}\label{the 1-form Lambda}
	\lambda^*(\varpi)=I\omega^1+J\omega^3,
	\end{equation}
	and a Riemannian metric $g$ with area form $dA$ and Gauss curvature $G$ determined only by the Finsler structure $F$. More precisely
	\begin{equation}\label{metric on Lambda}
	\begin{split}
	&\lambda^*(g)=\omega^1\otimes \omega^1+\omega^3\otimes \omega^3\\
	&\lambda^*(dA)=\omega^1\wedge \omega^3
	\end{split}
	\end{equation}
	and 
	$$
	\lambda^*(G)=1-I^2-J^2-I_3+J_1.
	$$
\end{theorem}
\begin{proof}
	The proof is quite straightforward. From \eqref{K=1 invariance} it is clear that the one form $I\omega^1+J\omega^3$ on $\Sigma$ is invariant under the geodesic flow of $F$ and hence it descends on the manifold of geodesics $\Lambda$. 
	
	Moreover, by using the  equations 
	\eqref{K=1 Finsler structure eq 2} and \eqref{K=1 Finsler Bianchi}
	it can be easily shown that the coframe
	\begin{equation}
	\alpha^1:=\omega^1,\ \alpha^2=\omega^3,\ \alpha^3=	\lambda^*(\varpi)-\omega^2
	\end{equation}
	gives the Riemannian metric $g$ with sectional curvature $G$ as in the theorem. 
	$\qedd$
\end{proof}

Conversely, one has

\begin{theorem}
	Let $(\Lambda,g)$ be a geodesically amenable Riemannian surface with the moving coframe $(\alpha^1, \alpha^2,\alpha^3)$ on the unit sphere bundle  $U^g\Lambda=\Sigma$ and Gauss curvature $G$, and let us assume that there exists a one form $\varpi$ on $\Lambda$ satisfying the structure equation
	\begin{equation}\label{varpi struct eq}
	d\lambda^*(\varpi)=(G-1)\alpha^1\wedge\alpha^2.
	\end{equation} 
	Then, on the manifold of geodesics $M$ there exists a $K=1$ Finsler structure $F$ with the canonical coframe 
	\begin{equation}\label{induced Finsler coframe}
	\omega^1=\alpha^1,\ \omega^2=\lambda^*(\varpi)-\alpha^3,\ \omega^3=\alpha^2,
	\end{equation}
	and the invariants $I=\varpi_1$, $J=\varpi_2$, where $\lambda^*(\varpi)=\varpi_1\alpha^1+\varpi_2\alpha^2$.
\end{theorem}

The proof is again straightforward using equations  	\eqref{K=1 Finsler structure eq 2} and \eqref{K=1 Finsler Bianchi} and the structure equations of a Riemannian structure. 

Observe that in the constructions above we always identify the indicatrix bundle $\Sigma$ of the Finsler structure with the unit sphere bundle of the Riemannian metric $g$.


A quick analysis of the leaf quotient spaces shows the following.

\begin{enumerate}
	\item The $(M,F)$-geodesic foliation $\{\omega^1=0,\omega^3=0\}$
	coincides with the 
	$(\Lambda,g)$-indicatrix foliation $\{\alpha^1=0,\alpha^2=0\}$.
	\item The $(M,F)$-indicatrix foliation $\{\omega^1=0,\omega^2=0\}$ coincides with the foliation $\{\alpha^1=0,\alpha^3-\lambda^*(\varpi)=0\}$, called by Bryant, the $\varpi$-foliation on $(\Lambda,g)$.
\end{enumerate}

If we denote  by subscripts the directional derivatives, that is $df=f_{\theta 1}\theta^1+f_{\theta 2}\theta^2+f_{\theta 3}\theta^3$, for any smooth function $f:\Sigma\to \R$, then one can obtain the following result.
\begin{theorem}\label{thm: Zoll to Finsler}
	Let $(\Lambda=\Sph^2,g)$ be a Zoll surface with everywhere positive Gauss curvature $G>0$, and let $M$ be the manifold of oriented geodesics of  $g$. Then, there exists a unique Finsler structure with $K=1$ on $M$ with the moving coframe 
	\begin{equation}\label{omega from alpha2}
	\begin{split}
	& \omega^1:=\lambda^*(\sqrt{G})\ {\theta}^1   \\
	& \omega^2:=  -{\theta}^3 \\
	& \omega^3:=\lambda^*(\sqrt{G})\ {\theta}^2,
	\end{split}
	\end{equation}
where $(\theta^1,\theta^2,\theta^3)$ is the coframe of $g$. 

	The invariants of this Finsler structure are given by 
	\begin{equation}\label{I, J from G & theta}
	I:=\frac{1}{2}\lambda^*(\frac{{G}_{\theta 2}}{G^\frac{3}{2}}),\qquad J:=-\frac{1}{2}\lambda^*(\frac{{G}_{\theta 1}}{G^\frac{3}{2}}).
	\end{equation}
\end{theorem}

\begin{proof}
	The idea is to consider a conformal change of $g$ with a 
      function $u$ and to determine this $u$ by using \eqref{varpi struct eq}.
	
	Indeed, the if we denote by $(\alpha^1,\alpha^2,\alpha^3)$ the coframe of the conformal metric $u^2g$, then 
	$$
	\alpha^1=u\theta^1,\ \alpha^2=u\theta^2,\ \alpha^3=\theta^3 -*d(\log u),
	$$
	where $*d(\log u)=-\frac{u_{\theta 2}}{u}\theta^1+\frac{u_{\theta 1}}{u}\theta^2$.
	
	It results that the coframe \eqref{induced Finsler coframe} reads
	\begin{equation}\label{induced Finsler coframe 1}
		\omega^1=u\theta^1,\ \omega^2=\lambda^*(\varpi)+*d(\log u)-\theta^3,\ \omega^3=u\theta^2,
	\end{equation}
	where $\lambda^*(\varpi)=u(I\theta^1+J\theta^2)$.
	
	An obvious choice is 
	\begin{equation}
	\lambda^*(\varpi)=-*d(\log u)
	\end{equation}
	in which case the coframe \eqref{induced Finsler coframe 1} simplifies to
	\begin{equation}\label{induced Finsler coframe 2}
	\omega^1=u\theta^1,\ \omega^2=-\theta^3,\ \omega^3=u\theta^2.
	\end{equation}
	
	Now we can use \eqref{varpi struct eq} to determine $u$, or equivalently, just to compute the structure equations of the coframe \eqref{induced Finsler coframe 2}. However, note that the function $G$ in \eqref{varpi struct eq} is the curvature of the metric $u^2g$ . Indeed, observe that $(\omega^1,\omega^2,\omega^3)$ in \eqref{induced Finsler coframe 2} satisfy the structure equations of a Finsler structure \eqref{K=1 Finsler structure eq 2} if and only if
	$$
	I=\frac{u_{\theta 2}}{u^2},\ u^2=\lambda^*(G)	,\ and \ J=-\frac{u_{\theta 1}}{u^2},
	$$
	respectively, that is the proof is finished.
	$\qedd$
\end{proof}

\begin{remark}
	\begin{enumerate}
		\item 	The construction above gives a non-Riemannian Finsler structure of $K=1$ on $M$ if and only if  both directional derivatives of $G$ with respect to $\theta^1$ and $\theta^2$ are non-vanishing functions on $\Lambda$. 
		\item The Finsler structure constructed in Theorem \ref{thm: Zoll to Finsler} satisfies the extra condition 
		\begin{equation}\label{extra condition}
			I_1+J_3=0,
		\end{equation}
		where subscripts are directional derivatives with respect to $\omega^1$, $\omega^3$. Indeed, if for an arbitrary smooth function $f$ on $\Sigma$, we denote $df=f_1\omega^1+f_2\omega^2+f_3\omega^3$, then 
		$$
		f_{\theta^1}=\lambda^*(\sqrt{G})f_1,\ 
f_{\theta^2}=\lambda^*(\sqrt{G})f_3,\ f_{\theta^3}=-f_2.		$$
	\end{enumerate}

And hence $I_{\theta 1}+J_{\theta 2}=\sqrt{G}(I_1+J_3)$. On the other hand, observe that \eqref{I, J from G & theta} implies $I_{\theta 1}+J_{\theta 2}=0$ and hence \eqref{extra condition} follows. 

\end{remark}

Conversely, we can start from a Finsler surface $(M,F)$ of Zoll type on $M=\Sph^2$, that is a Finsler surface all of whose geodesics are closed and have the same length $2\pi$, with coframe $(\omega^1,\omega^2,\omega^3)$ on the indicatrix bundle $\Sigma$ and invariants $I,J,K=1$, and construct a Zoll metric on the manifold of geodesics $\Lambda=\Sph^2$. 

Indeed, let us consider on $\Sigma$ the directional PDE system with respect to $(\omega^1,\omega^2,\omega^3)$
\begin{equation}\label{directional PDE G}
\begin{split}
& \hat{G}_1+2J\hat{G}=0\\
& \hat{G}_3-2I\hat{G}=0
\end{split}
\end{equation}
for an unknown function $\hat{G}:\Sigma\to \R$. This PDE system has solutions if and only if 
$	I_1+J_3=0$, and in this case, the solution depends on a constant, only. 

More precisely, we need to consider the global existence and the positiveness of the solution. This can be seen as follows. Observe that the condition 
$	I_1+J_3=0$ implies $d(-J\omega^1+I\omega^3)=0$, and since $\Sigma$ is diffeomorphic to $\mathbb{RP}^3$, the first de Rham cohomology class is zero, hence there exists a function $\hat{\rho}:\Sigma\to\R$ such that
$$
d\hat{\rho}=-J\omega^1+I\omega^3.
$$

Since $\hat{e}_2(\hat{\rho})=0$ it follows that there exists a function $\rho:\Lambda\to\R$ such that $\lambda^*(\rho)=\hat{\rho}$, and hence the solution we seek will be given by $G=e^{2\rho}$. 

We obtain 
\begin{theorem}\label{thm: Finsler to Zoll}
	Let $(M=\Sph^2,F)$ be a Finsler surface of Zoll type with coframe $(\omega^1,\omega^2,\omega^3)$ on the indicatrix bundle $\Sigma$ and invariants $I,J,K=1$, that satisfies the condition $	I_1+J_3=0$. 
	
	Then the manifold of geodesics $\Lambda=\Sph^2$ can be endowed with a Riemannian metric $g$ with the coframe 
	\begin{equation}
	\begin{split}
	\theta^1&=\hat{G}^{-\frac{1}{2}}\omega^1\\
	\theta^2&=\hat{G}^{-\frac{1}{2}}\omega^3\\
	\theta^3&=-\omega^2,
	\end{split}
	\end{equation}
 and the Gauss curvature $\hat{G}$, 	where $\hat{G}$ is the solution of the directional PDE \eqref{directional PDE G}. In fact $(\Lambda,g)$ is a Zoll manifold. 
\end{theorem}

Indeed, it is clear from our construction that the geodesics of the Riemannian manifold $(\Lambda,g)$ are closed, but the fact that they also have the same length is not proved yet. However, this follows immediately from Wadsley Theorem (see \cite{Be}, Theorem 7.12, p. 183). This common length can be arranged to be $2\pi$ by choosing an appropriate constant factor in $\hat G$. 

\begin{remark}
	The Theorems \ref{thm: Zoll to Finsler} and  \ref{thm: Finsler to Zoll} are reciprocal each other, leading to the following important question:
	
	{\it There is an one-to-one correspondence between $G>0$ Zoll metrics and $K=1$ Finsler manifolds?}

The answer is positive, there is an one-to-one correspondence between $\mathcal K$-Cartan structures (in this case $\mathcal K=G$) and $K=1$ Finsler structures up to diffeomorphism and conformal equivalence (see \cite{SSP2014}, section 7 for details). 


\end{remark}




\begin{remark}

Taking into account the invariance formulas \eqref {K=1 invariance}, if we denote by $\xi_t: \Sigma\to \Sigma$, $t\in \R$, the geodesic flow of $F$, that is the flow of $\hat{e}_2$, then it is trivial to see that
\begin{equation}\label{geod flow frame}
\begin{split}
& \xi_{t,*}(\hat{e}_1)=\cos t\ \hat{e}_1-\sin t\  \hat{e}_3\\
& \xi_{t,*}(\hat{e}_2)= \hat{e}_2\\
& \xi_{t,*}(\hat{e}_3)=\sin t\ \hat{e}_1+\cos t \  \hat{e}_3.
\end{split}
\end{equation}

 It is obvious from the construction presented above that
 the $K=1$ Finsler structure $F$ induced by a positively curved Zoll metric has all geodesics closed and of same length $2\pi$.
	 
	 	Moreover, the unit speed geodesics of this Finsler structure, emanating from a fixed point $p\in M$, intersect 
		at the distance $\pi$ in the same point $q$.

Indeed, let us consider a parametrization $s\mapsto v_s$ of the indicatrix $\Sigma_p\subset T_pM$, $s\in \Sph^1$, $F(v_s)=1$. It is clear from the general construction that $\hat{e}_3$ is tangent to the indicatrix, that is $\dfrac{d}{ds}v_s=-\hat{e}_3|_{v_s}$. 
	
	By means of the geodesic flow $\xi_t$ introduced above, $\xi_\pi(v_s)\in TM$, and moreover we have
	$$
	\frac{d}{ds}\xi_\pi(v_s)=\xi_{\pi,*}(\dfrac{d}{ds}v_s)=-\hat{e}_3|_{\xi_\pi(v_s)},
	$$
	where we have used \eqref{geod flow frame}.
	
	Therefore, there exists a point $q\in M$ such that 
	$\xi_\pi(v_s)\in T_qM$, for any parameter value $s$, and the statement follows. 
	
\end{remark}


\section{The geometry of a Zoll surface of revolution}\label{sec: Zoll surfaces}
In this section we review the basic facts on the geometry of a Zoll surface of revolution needed in the next sections. Our main reference is \cite{Be}.
\subsection{Geodesics on a Zoll surface}\label{subsec: Zoll geod}
Let us consider the Zoll surface of revolution $(\Lambda,g)=(\Sph^2,g)$ with the coordinates and metric $g$ described in Introduction (see \eqref{Zoll metric }).



We describe now the geodesics of $(\Lambda,g)$ in terms of the Hamiltonian formalism. We consider the local coordinates $(r,\theta;\xi_1,\xi_2)$ on the cotangent space $T^*\Lambda$, and the Hamiltonian function
$$
2E=\frac{\xi_1^2}{[1+h(\cos r)]^2 }+\frac{\xi_2^2}{\sin^2r}.
$$

Since we have a surface of revolution, $F=\xi_2$ is a first integral, i.e. it can be checked by direct computation that the Poisson bracket vanishes $\{E,F\}=0$.

We consider now the geodesics such that $2E=1$ and $\xi_2=c\in[-1,1]$, that is
\begin{equation}\label{eq (3.1)}
\begin{cases}
\frac{\xi_1^2}{[1+h(\cos r)]^2 }+\frac{\xi_2^2}{\sin^2r}=1\\
\xi_2=c.
\end{cases}
\end{equation}
It follows
\begin{equation}\label{eq (3.2)}
\xi_1=\pm [1+h(\cos r)] \sqrt{1-\frac{c^2}{\sin^2r}}
\end{equation}

We obtain the geodesic (flow) $t\mapsto (r(t),\theta(t);\xi_1(t),\xi_2(t))$ given by
\begin{equation}\label{Zoll geod}
\begin{cases}
\frac{dr}{dt}=\frac{\partial E}{\partial \xi_1}=\frac{\xi_1}{[1+h(\cos r)]^2 }=\pm\frac{1}{1+h(\cos r)}\sqrt{1-\frac{c^2}{\sin^2r}}\\
\frac{d\theta}{dt}=\frac{\partial E}{\partial \xi_2}=\frac{\xi_2}{\sin^2r}=\frac{c}{\sin^2r}.
\end{cases}
\end{equation}

Observe that 
$$
\frac{d\theta}{dt}\cdot \sin^2r=c
$$
is called the {Clairaut constant}. 



By \eqref{eq (3.1)} or \eqref{eq (3.2)} we have $|c|\le \sin r$.
Put
\begin{equation*}
r_c=\arcsin c\in [0,\pi/2].
\end{equation*}
Then the range of $r(t)$ along the geodesic is $[r_c,\pi-r_c]$. Since $h$ is an odd function, we have
\begin{gather*}\label{ds over dr}
\int_{r_c}^{\pi-r_c}\frac{dt}{dr}=\pm
\int_{r_c}^{\pi-r_c}\frac{\sin r[1+h(\cos r)]}{\sqrt{\sin^2 r-\sin^2 r_c}}\,dr=\pm
 \pi\\
\int_{r_c}^{\pi-r_c}\frac{d\theta}{dr}=\pm
\int_{r_c}^{\pi-r_c}\frac{\sin r_c[1+h(\cos r)]}{\sin r\sqrt{\sin^2 r-\sin^2 r_c}}\,dr=\pm\pi,
\end{gather*}
which indicate that all geodesics are closed
and have length $2\pi$.

Let us now describe the frame $(\hat{v}_1,\hat{v}_2,\hat{v}_3)$
on the unit tangent bundle $\Sigma=U\Lambda$
by means of the coordinates $(r,\theta)$ on
$\Lambda$. To do so, we use the decomposition
\begin{equation*}
T_v\Sigma=H_v+V_v\qquad (v\in\Sigma)
\end{equation*}
of the tangent space into its horizontal and vertical parts and the natural identifications
\begin{equation*}
H_v\simeq T_{\lambda(v)}\Lambda,\quad V_v\simeq v^\perp\subset T_{\lambda(v)}\Lambda.
\end{equation*}
If $X\in T_v\Sigma$ is decomposed to the sum of $X_1\in T_{\lambda(v)\Lambda}$ (horizontal
part) and $X_2\in v^\perp$ (vertical part),
then we will write it as
\begin{equation*}
X=\begin{pmatrix}X_1\\X_2\end{pmatrix}.
\end{equation*}

We define the orientation on $\Lambda$ so that $\partial/\partial r, \partial/\partial \theta$ is positive in this order. Let $n(t)$
be the unit normal vector to $\dot\gamma(t)$
such that $\dot\gamma(t), n(t)$ is positive
in this order. Then they are described as
\begin{equation}\label{flame1}
\begin{gathered}
\dot\gamma(t)=\epsilon\frac1{1+h(\cos r)}\sqrt{1-\frac{c^2}{\sin^2 r}}\frac{\partial}{\partial r}+ \frac{c}{\sin^2 r}\frac{\partial}{\partial\theta},\\
n(t)=\frac{-c}{\sin r[1+h(\cos r)]}\frac{\partial}{\partial r}+\epsilon
\frac{\sqrt{\sin^2 r-c^2}}{\sin^2 r}\frac{\partial}{\partial \theta},
\end{gathered}
\end{equation}
where $\epsilon=\pm 1$, and the vector fields
$\hat{v}_1,\hat{v}_2, \hat{v}_3$ at $\dot\gamma(t)\in\Sigma$
are described as follows:
\begin{equation}\label{frame2}
\hat{v}_1=\begin{pmatrix}n(t)\\0\end{pmatrix},\qquad
\hat{v}_2=\begin{pmatrix}\dot\gamma(t)\\0\end{pmatrix},\qquad
\hat{v}_3=\begin{pmatrix}0\\n(t)\end{pmatrix}.
\end{equation}
For the sake of convinience, we will write $\hat{\gamma}(t)$ instead of
$\dot{\gamma}(t)$ when it represents the point of $\Sigma$. Also, we will denote by $[\gamma]$ the corresponding point on the manifold of geodesic, $M$.

\subsection{Jacobi fields}\label{subsec: Zoll surface Jacobi fields}

It is known that the tangent space at a point $\gamma$ to the manifold of geodesics of a Zoll surface is given by the space of Jacobi fields along the geodesic $\gamma$.

We start by decomposing the Killing vector field $\frac{\partial}{\partial \theta}$ along $\gamma(t)$ as
\begin{equation}\label{decom d theta}
\frac{\partial}{\partial \theta}|_{\gamma(t)}=c\cdot \dot{\gamma}(t)\pm \sqrt{\sin^2 r-c^2}\cdot n(t).
\end{equation}

Let us define the Jacobi field
\begin{equation}
Y(t):=y(t)n(t),\quad y(t)=\pm\sqrt{\sin^2r(t)-c^2}
\end{equation}
along $\gamma(t)$, and compute
\begin{equation}\label{comput y'}
y'(t)=\frac{\cos r}{1+h(\cos r)}.
\end{equation}

Observe that
\begin{equation}
Y(0)=0,\  i.e.\  y(0)=0.
\end{equation}

If we take one more derivative, we can compute the Gauss curvature $G(r)$ of this Zoll metric by means of the Jacobi equation, that is
\begin{equation}
y''(t)=\mp \Bigl[
1-\frac{\cos r h'(\cos r)}{1+h(\cos r)}
\Bigr]
\frac{\sqrt{\sin^2r-c^2}}{[1+h(\cos r)]^2}
=-G(\gamma(t))y(t),
\end{equation}
where 
\begin{equation}\label{Gauss curv for Zoll surf}
\hat{G}(r)=G(r)=\frac{1}{[1+h(\cos r)]^3}\Bigl[ 1+h(\cos r) -\cos r\cdot h'(\cos r) 
\Bigr],
\end{equation}
that is usual formula for the Gauss curvature of a Zoll metric (see for instance \cite{Be}, p. 105).

\bigskip

\begin{figure}[h]
	\begin{center}
		
		\setlength{\unitlength}{1cm}
		\begin{picture}(0,5)(0,-2.5)
		\qbezier(-2.5,0)(-2.465,1.25)(-1.85,1.85)
		\qbezier(0,2.5)(-1.25,2.465)(-1.85,1.85)
		\qbezier(2.5,0)(2.465,1.25)(1.85,1.85)
		\qbezier(0,2.5)(1.25,2.465)(1.85,1.85)
		\qbezier(-2.5,0)(-2.465,-1.25)(-1.85,-1.85)
		\qbezier(0,-2.5)(-1.25,-2.465)(-1.85,-1.85)
		\qbezier(2.5,0)(2.465,-1.25)(1.85,-1.85)
		\qbezier(0,-2.5)(1.25,-2.465)(1.85,-1.85)

		\qbezier(-1.9,1.624807681)(0,1.1)(1.9,1.624807681)
		\qbezier(-2.1,-1.356465997)(0,-1.8)(2.1,-1.356465997)
		
		\qbezier[65](0,2.5)(1,0)(0,-2.5)
		{\color{red}
			\qbezier(-1,-1)(-0.8,1.22)(0.33,1.38)
			\qbezier(0.33,1.38)(1.4,1.22)(1.5,0)
			\put(1.5,0){\vector(0,-1){0}}
			\put(-0.86,0){\vector(0,1){0}}
			\begingroup
			\scriptsize
			\put(-1.8,-1.2){$\gamma(t)=(r(t),\theta(t))$}
			\put(-1.9,0){$\gamma(t)\searrow$}
			\put(1.5,0.3){$\gamma(t)\nearrow$}
			\put(0,1){$\gamma(0)$}
			\endgroup
		}

		\put(0.8,1.24){\vector(1,1){0.5}}
		\put(1.1,1){\vector(1.3,1){0.5}}
		\put(-0.38,1){\vector(-1,1){0.5}}
		\put(0.33,1.38){\vector(0.2,1){0.3}}
		
		\begingroup
		\scriptsize
		\put(-1,1.5){$n$}
		\put(1.2,1.8){$n$}
		\put(1.6,1.3){$n$}
		\put(0.7,2.8){$n$}
		\put(-2,2.8){$(r,\theta)=(0,\theta_0)$}
		\put(-0.5,-2.8){$(\pi,\theta_0)$}
		\put(2,1.7){$r=r_c=\arcsin c$}
		\put(2.2,-1.6){$r=\pi-r_c$}
		\put(2.4,-2){$=\pi-\arcsin c$}
		\endgroup
		
		\end{picture}
		\caption{A geodesic $\gamma$ on the Zoll sphere.}	\label{fig1}
	\end{center}
\end{figure}
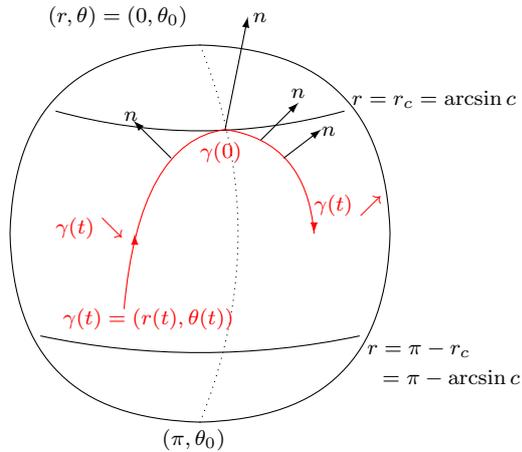

Let us consider the parameter $t$ on the geodesic $\gamma$ such that at $t=0$, $\gamma(0)=r_c:=\arcsin c$, see Figure \ref{fig1}.

We will construct in the following two {\it normalized} Jacobi fields $Y_1$ and $Y_2$ along $\gamma(t)$. Here normalized means that we will determine these Jacobi fields subject to the initial conditions
\begin{equation}\label{Jacobi IC}
\begin{cases}
Y_1(0)=0,\quad Y'_1(0)=n(0)\\
Y_2(0)=n(0),\quad Y'_2(0)=0.
\end{cases}
\end{equation}

\begin{proposition}\label{prop:Jacobi fields Y1 Y2}
The normalized Jacobi fields
along $\gamma(t)$ are given by
$Y_i(t):=y_i(t)n(t)$, $i=1,2$, where
\begin{equation}\label{y_1,y_2 formulas}
\begin{split}
& y_1(t)=\frac{1+h(\cos r_c)}{\cos r_c}\cdot y(t)=
\pm \frac{1+h(\cos r_c)}{\cos r_c} \sqrt{\sin^2r-c^2}
\\
& y_2(t)=\frac{1}{y'_1(t)}-y_1(t)\int_0^t
\frac{K(\gamma(s))}{[y'_1(s)]^2}ds
\end{split}
\end{equation}
\end{proposition}
\begin{proof}
We can check by direct computation that these $Y_i(t)$ satisfy the initial conditions \eqref{Jacobi IC} by using the computation for $y(t)$.

\end{proof}


\section{The manifold of geodesics}\label{sec: manif of geod}
\subsection{The embedding of $\Sigma$ into the tangent space $TM$}\label{subsec: the embedding}

\quad To define the desired Finsler metric on $M$,
we will embed $\Sigma$ into $TM$ so that $\Sigma$ is realized as the indicatrix bundle.
Since the vector field $\hat{e}_2$ will be
the generator of the geodesic flow, this embedding should be given by (compare with Bryant [3])
\begin{equation*}
\iota:\Sigma\to TM,\qquad u\to \iota(u)=\pi_{*,u}(\hat{e}_2).
\end{equation*}

Let us describe
\begin{equation*}
\iota(\hat{\gamma}(t))=a(t)Y_1+b(t)Y_2\in T_{[\gamma]}M,
\end{equation*}
where $Y_i$ are the tangent vectors to  $M$
at $[\gamma]$ which correspond to the Jacobi
fields $Y_i(t)$ along $\gamma(t)$ defined before. One obtains:

\medskip
\begin{proposition}\quad  $a(t)=-y_2(t),\quad b(t)=y_1(t)$.
\end{proposition}
\begin{proof}\quad
Since $\hat{e}_2=-\hat{v}_3$, we have
\begin{equation*}
a(t)\begin{pmatrix}Y_1(t)\\Y'_1(t)\end{pmatrix} + b(t)\begin{pmatrix}Y_2(t)\\Y'_2(t)\end{pmatrix}=-\begin{pmatrix}0\\n(t)\end{pmatrix},
\end{equation*}
that is, in scalar form,
\begin{equation*}
a(t)\begin{pmatrix}y_1(t)\\y'_1(t)\end{pmatrix} + b(t)\begin{pmatrix}y_2(t)\\y'_2(t)\end{pmatrix}=-\begin{pmatrix}0\\1\end{pmatrix}.
\end{equation*}
Thus the proposition follws.

$\qedd$
\end{proof}

\subsection{Coordinates on the manifold of geodesics}\label{subsec:coordinates on M}

In this section we introduce a system of coordinates 
$$(R,\Theta)\in (-\frac{\pi}{2},\frac{\pi}{2})\times \R\slash 2\pi \mathbb Z 
$$ 
on the manifold of geodesics $M$, which is a 2-dimensional manifold, as follows. 

The base manifold of the Zoll surface $(\Lambda,g)$ is $\Sph^2$ with the usual spherical coordinates 
$(r,\theta)\in [0,\pi]\times \R\slash 2\pi \mathbb Z$. However, in the present context we will consider negative values of the coordinate $r$ as well, that is, for a fixed $\theta=\theta_0$, the half meridian  
$\{(r,\theta):-\frac{\pi}{2}<r<\frac{\pi}{2},\theta=\theta_0\}$ is included in the northern hemisphere, see Figure \ref{fig2}.

If we consider a unit length tangent vector in the direction $\frac{\partial}{\partial \theta}$ at $(r_0,\theta_0)$, for some $0<r_0<\frac{\pi}{2}$,  it is easy to see that this tangent vector can be smoothly extended to a vector field on $\Lambda$ for any $-\frac{\pi}{2}<r\leq 0$.

In other words, for the same non-oriented geodesic $\gamma_0$ with initial conditions $\gamma_0(0)=(r_0,\theta_0)$, $\dot{\gamma}(0)=\frac{1}{\sin r_0}$, we need to make difference if the initial velocity is in the direction of $\frac{\partial}{\partial \theta}$ or opposite direction. Observe that in terms of Clairaut constants, these geodesics correspond to $c_0>0$ and $c_0<0$, and therefore we will denote these oriented geodesics by $\gamma_0^+$ and $\gamma_0^-$, respectively. On the manifold of {\it oriented} geodesics $M$ they will give different points $[\gamma_0^+]$ and $[\gamma_0^-]$.

In order to introduce local coordinates on the manifold of geodesics, it is usefull to make the following convention. In order to make distinction between $\gamma_0^+$ and $\gamma_0^-$, we will identify $\gamma_0^-$ with the geodesic on $(\Lambda,g)$ with initial point $(-r_0, \theta+\pi)$ and Clairaut constant $-c_0$.


\bigskip

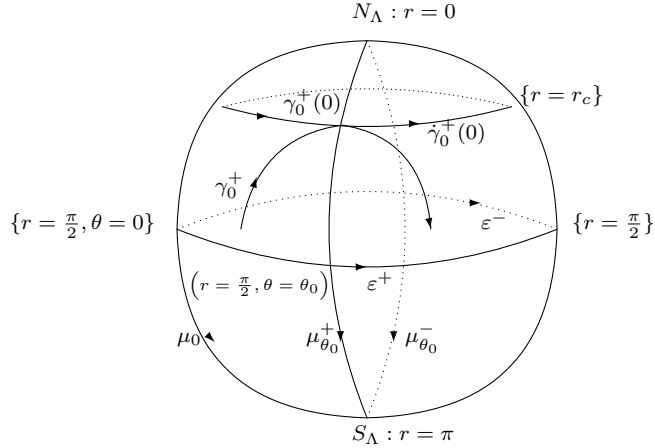
\begin{figure}[h]
	\begin{center}
		
		\setlength{\unitlength}{1cm}
		\begin{picture}(0,5)(0,-2.5)
		\qbezier(-2.5,0)(-2.465,1.25)(-1.85,1.85)
		\qbezier(0,2.5)(-1.25,2.465)(-1.85,1.85)
		\qbezier(2.5,0)(2.465,1.25)(1.85,1.85)
		\qbezier(0,2.5)(1.25,2.465)(1.85,1.85)
		\qbezier(-2.5,0)(-2.465,-1.25)(-1.85,-1.85)
		\qbezier(0,-2.5)(-1.25,-2.465)(-1.85,-1.85)
		\qbezier(2.5,0)(2.465,-1.25)(1.85,-1.85)
		\qbezier(0,-2.5)(1.25,-2.465)(1.85,-1.85)
		\qbezier(-1.9,1.624807681)(0,1.1)(1.9,1.624807681)
		\qbezier[65](-1.9,1.624807681)(0,2.1)(1.9,1.624807681)
		
		\qbezier(2.5,0)(0,-1)(-2.5,0)
		\qbezier[65](2.5,0)(0,1)(-2.5,0)

		\qbezier(0,2.5)(-1,0)(0,-2.5)
		\qbezier[65](0,2.5)(1,0)(0,-2.5)
		\qbezier(-1.66,0)(-1.46,1.22)(-0.33,1.38)
		\qbezier(-0.33,1.38)(0.74,1.22)(0.84,0)
		
		\put(0.84,0){\vector(0.2,-1){0}}
		\put(-1.45,0.7){\vector(1,3){0}}
		\put(0.7,1.4){\vector(1,0){0}}
		\put(-1.3,1.5){\vector(1,0){0}}
		\put(-0.35,-1.5){\vector(0,-1){0}}
		\put(0.35,-1.5){\vector(0,-1){0}}
		\put(-2,-1.5){\vector(1,-1){0}}
		
		\put(0,-0.5){\vector(1,0){0}}
		\put(1.5,0.35){\vector(1,0){0}}

		\begingroup
		\scriptsize
		\put(-0.2,-2.8){$S_\Lambda :r=\pi$}
		\put(-0.2,2.8){$N_\Lambda :r=0$}
		\put(2,1.7){$\{r=r_c\}$}
		\put(2.7,0){$\{r=\frac{\pi}{2}\}$}
		\put(-4.7,0){$\{r=\frac{\pi}{2},\theta=0\}$}
		
		\put(1.5,0){$\varepsilon^-$}
		\put(0,-0.8){$\varepsilon^+$}
		
		\put(-2,0.5){$\gamma_0^+$}
		\put(-1.1,1.6){$\gamma_0^+(0)$}
		\put(0.8,1.2){$\dot{\gamma}_0^+(0)$}
		
		\put(-2.5,-1.5){$\mu_0$}
		\put(-0.8,-1.5){$\mu_{\theta_0}^+$}
		\put(0.5,-1.5){$\mu_{\theta_0}^-$}

		\endgroup
		
		\begingroup
		\tiny
		\put(-2.35,-0.8){$\left(r=\frac{\pi}{2},\theta=\theta_0\right)$}

		\endgroup
		\end{picture}
		\caption{Coordinates on the Zoll sphere.}	\label{fig2}
	\end{center}	
\end{figure}

From topological reasons it is known that the manifold of geodesics is a sphere, so we will start by introducing the local coordinates $(R,\Theta)$ on $M=\Sph^2$ except two points, the poles on $M$, that correspond to the two equators $\{r=\frac{\pi}{2}\}$ of different orientations $\ve^+$ and $\ve^-$, that is geodesics with Clairaut constants $c=\pm 1$. 

For a geodesic $\gamma(t)=(r(t),\theta(t))$ of $(\Lambda,g)$, with Clairaut constant $c\neq \pm 1$, we define the local coordinates $(R,\Theta)$ of the corresponding point $[\gamma]\in M$ as follows
\begin{equation}
(R,\Theta)=
\begin{cases}
(r(0),\theta(0))\textrm{ , if } c>0\\
(-r(0),\theta(0)+\pi)\textrm{ , if } c<0\\
(0,\theta(\dot \gamma(0))-\frac{\pi}{2})\textrm{ , if } c=0,
\end{cases}
\end{equation}
(see figures \ref{fig3a}, \ref{fig3b}).


\begin{figure}[h]
	\begin{center}
		
		\setlength{\unitlength}{1cm}
		\begin{picture}(0,5)(0,-2.5)
		\qbezier(-2.5,0)(-2.465,1.25)(-1.85,1.85)
		\qbezier(0,2.5)(-1.25,2.465)(-1.85,1.85)
		\qbezier(2.5,0)(2.465,1.25)(1.85,1.85)
		\qbezier(0,2.5)(1.25,2.465)(1.85,1.85)
		\qbezier(-2.5,0)(-2.465,-1.25)(-1.85,-1.85)
		\qbezier(0,-2.5)(-1.25,-2.465)(-1.85,-1.85)
		\qbezier(2.5,0)(2.465,-1.25)(1.85,-1.85)
		\qbezier(0,-2.5)(1.25,-2.465)(1.85,-1.85)
		\qbezier(-1.9,1.624807681)(0,1.1)(1.9,1.624807681)
		\qbezier[65](-1.9,1.624807681)(0,2.1)(1.9,1.624807681)
		\qbezier(-2.1,-1.356465997)(0,-2)(2.1,-1.356465997)
		\qbezier[65](-2.1,-1.356465997)(0,-1)(2.1,-1.356465997)
		
		\qbezier(2.5,0)(0,-1)(-2.5,0)
		\qbezier[65](2.5,0)(0,1)(-2.5,0)

		\qbezier(0,2.5)(-1,0)(0,-2.5)
		\qbezier[65](0,2.5)(1,0)(0,-2.5)
		
		\begingroup
		\scriptsize
		\put(0,2.8){$N_M:R=\frac{\pi}{2}$}
		\put(0,-2.8){$S_M:R=-\frac{\pi}{2}$}
		\put(-5,-1){$\{\Theta=0\}\equiv\left[\mu_{-\frac{\pi}{2}}\right]$}
		\put(2.6,0){$\{R=0\}$}
		\put(-0.9,1.6){$[\gamma^+_0]$}
		\put(-0.9,-2){$[\gamma^-_0]$}
		\put(-1.2,-0.8){$[\mu^+_{\theta_0}]$}
		\put(0.5,0.6){$[\mu^-_{\theta_0}]$}
		\endgroup

		\end{picture}
		\caption{The manifold of geodesics.}\label{fig3a}
	\end{center}
\end{figure}
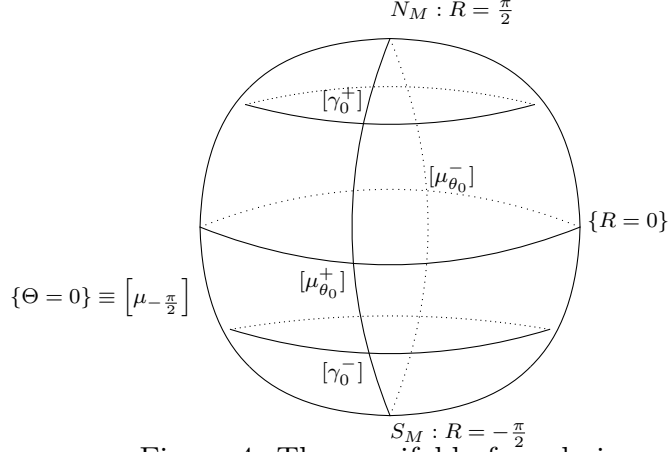

\begin{figure}[h]
	\begin{center}
		
		\setlength{\unitlength}{1.2cm}
		\begin{picture}(0,5)(0,-2.5)
		\qbezier(-2.5,0)(-2.465,1.25)(-1.85,1.85)
		\qbezier(0,2.5)(-1.25,2.465)(-1.85,1.85)
		\qbezier(2.5,0)(2.465,1.25)(1.85,1.85)
		\qbezier(0,2.5)(1.25,2.465)(1.85,1.85)
		\qbezier(-2.5,0)(-2.465,-1.25)(-1.85,-1.85)
		\qbezier(0,-2.5)(-1.25,-2.465)(-1.85,-1.85)
		\qbezier(2.5,0)(2.465,-1.25)(1.85,-1.85)
		\qbezier(0,-2.5)(1.25,-2.465)(1.85,-1.85)
		
		\qbezier(-2.5,0)(0,-1.2)(2.5,0)
		\qbezier(0,2.5)(-1,0)(0,-2.5)
		\qbezier[65](-2.5,0)(0,1.4)(2.5,0)
		
		\qbezier(-1.9,1.624807681)(0,1.1)(1.9,1.624807681)
		
		\qbezier(0,0.2)(0.33,2.1)(1,1)
		\put(1,1){\vector(1,-1){0}}
		
		\qbezier(0.5,0.2)(0.83,2.2)(1.5,1)
		\put(1.5,1){\vector(1,-1){0}}
		
		\qbezier(-0.8,0.2)(-0.17,2)(0.5,1)
		\put(0.5,1){\vector(1,-1){0}}
		
		\qbezier(-1.5,0.2)(-0.83,2)(0,1)
		\put(0,1){\vector(1,-1){0}}
		
		\qbezier(-1.7,0.7)(-0.83,1.9)(0,1.5)
		\put(0,1.48){\vector(1,-1){0}}
		
		\qbezier(-1.9,0.9)(-0.83,2.2)(0.5,1.7)
		\put(0.5,1.7){\vector(2,-1){0}}
		
		\qbezier(-2.1,1.1)(-0.83,2.4)(0.9,1.9)
		\put(0.9,1.9){\vector(2,-1){0}}
		
		\begingroup
		\scriptsize
		\put(0,2.7){$N_\Lambda :r=0$}
		\put(0,-2.7){$S_\Lambda :r=\pi$}
		\put(2,1.6){${r=r_c}$}
		\put(2.7,0){$\{r=\frac{\pi}{2}\}$}
		\put(-1.7,0){$\gamma_0^+$}
		\put(-1,0){$\gamma_1^+$}
		\put(0,0){$\gamma_2^+$}
		\put(0.5,0){$\gamma_3^+$}
		\put(-1.8,0.5){$\sigma_1^+$}
		\put(-2,0.7){$\sigma_2^+$}
		\put(-2.2,0.9){$\sigma_3^+$}
		\endgroup

		\end{picture}
		\caption{A geodesic variation on the Zoll sphere.}\label{fig3b}
	\end{center}
\end{figure}
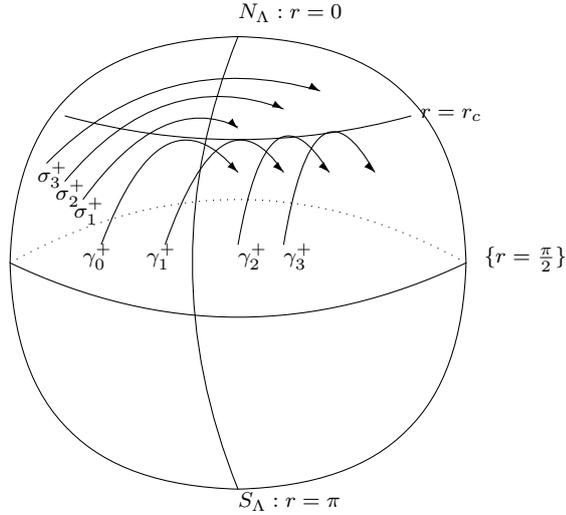

\begin{remark}
	 By considering a geodesic variation of the base geodesic $\gamma_0^+$ with the variation vector fields $Y_1$ and $Y_2$, we obtain the variations $\gamma_s^+$ and $\sigma_s^+$,respectively. By the coordinates system above they will give points on the parallel $R=r_c$ and the meridian $\Theta=\theta_0$.
\end{remark}

\begin{remark}
	 Let us consider geodesics on $(\Lambda,g)$ that start from the Nothern Pole $p_0:=N_\Lambda$, that is meridians. 
	 Let us also consider an orthonormal basis $v_0$, $v_1$ of $T_{p_0}\Lambda$ positive oriented (we use here $\Sph^2$ orientation) defined as follows. For a geodesic $\gamma_\theta^+(t)=(r(t),\theta)$, with $\theta$ constant, obtained from the variation of $\gamma_0^+(t)=(r(t),0)$, $t\in (0,\pi)$, with variation vector field $Y_2$ passing through the pole $p_0$, that is a meridian, we have $\gamma_\theta^+(0)=p_0$. We put 
	 \begin{equation}\label{IV1}
	 \dot{\gamma}_\theta(0)=\cos\theta\cdot v_0+\sin\theta \cdot v_1.
	 \end{equation}  
	
	In order to induce a smooth manifold structure on $M$, we identify the geodesic from $p_0$ and initial velocity \eqref{IV1} with the geodesic with same initial point and initial velocity $J(\dot{\gamma}_\theta(0))$, where $J:T_{p}\Lambda\to T_p\Lambda$ is linear mapping that gives the positive rotation by $\frac{\pi}{2}$ in any tangent plane to a point of $\Lambda$ (see Figure \ref{fig4a}, \ref{fig4b}). 
\end{remark}


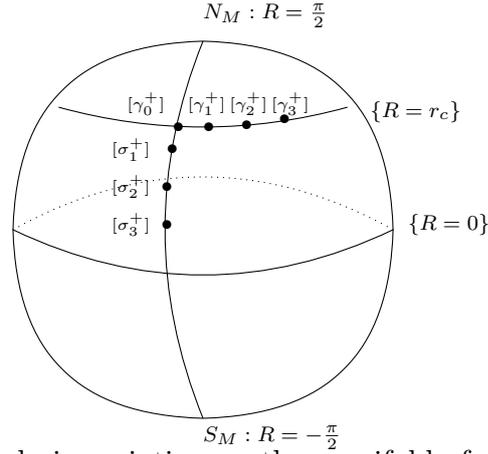
\begin{figure}[h]
	\begin{center}
		
		\setlength{\unitlength}{1cm}
		\begin{picture}(0,5)(0,-2.5)
		\qbezier(-2.5,0)(-2.465,1.25)(-1.85,1.85)
		\qbezier(0,2.5)(-1.25,2.465)(-1.85,1.85)
		\qbezier(2.5,0)(2.465,1.25)(1.85,1.85)
		\qbezier(0,2.5)(1.25,2.465)(1.85,1.85)
		\qbezier(-2.5,0)(-2.465,-1.25)(-1.85,-1.85)
		\qbezier(0,-2.5)(-1.25,-2.465)(-1.85,-1.85)
		\qbezier(2.5,0)(2.465,-1.25)(1.85,-1.85)
		\qbezier(0,-2.5)(1.25,-2.465)(1.85,-1.85)
		
		\qbezier(-2.5,0)(0,-1.2)(2.5,0)
		\qbezier(0,2.5)(-1,0)(0,-2.5)
		\qbezier[65](-2.5,0)(0,1.4)(2.5,0)
		
		\qbezier(-1.9,1.624807681)(0,1.1)(1.9,1.624807681)

		\begingroup
		\scriptsize
		\put(0,2.8){$N_M:R=\frac{\pi}{2}$}
		\put(0,-2.8){$S_M:R=-\frac{\pi}{2}$}
		\put(2.2,1.5){$\{R=r_c\}$}
		\put(2.7,0){$\{R=0\}$}
		\put(-0.55,0){$\bullet$}
		\put(-0.55,0.5){$\bullet$}
		\put(-0.48,1){$\bullet$}
		\put(-0.4,1.3){$\bullet$}
		\put(0,1.3){$\bullet$}
		\put(0.5,1.32){$\bullet$}
		\put(1,1.4){$\bullet$}
		\tiny
		\put(-1,1.6){$[\gamma^+_0]$}
		\put(-0.2,1.6){$[\gamma^+_1]$}
		\put(0.35,1.6){$[\gamma^+_2]$}
		\put(0.9,1.6){$[\gamma^+_3]$}
		\put(-1.2,1){$[\sigma^+_1]$}
		\put(-1.2,0.5){$[\sigma^+_2]$}
		\put(-1.2,0){$[\sigma^+_3]$}
		\endgroup
		
		\end{picture}
		\caption{A geodesic variation on the manifold of geodesics.}
	\end{center}
\end{figure}

\begin{figure}[h]
	\begin{center}
		
		\setlength{\unitlength}{1cm}
		\begin{picture}(0,5)(0,-2.5)
		\qbezier(-2.5,0)(-2.465,1.25)(-1.85,1.85)
		\qbezier(0,2.5)(-1.25,2.465)(-1.85,1.85)
		\qbezier(2.5,0)(2.465,1.25)(1.85,1.85)
		\qbezier(0,2.5)(1.25,2.465)(1.85,1.85)
		\qbezier(-2.5,0)(-2.465,-1.25)(-1.85,-1.85)
		\qbezier(0,-2.5)(-1.25,-2.465)(-1.85,-1.85)
		\qbezier(2.5,0)(2.465,-1.25)(1.85,-1.85)
		\qbezier(0,-2.5)(1.25,-2.465)(1.85,-1.85)
		\put(0,1.3){\circle{1.5}}
		
		\qbezier(0,-2.5)(-0.65,-2.1)(-0.65,-0.65)
		\qbezier(-0.65,-0.65)(-0.5,1.5)(1,2.291287847)
		
		\qbezier(-0.21,1)(0.145,0.7)(0.5,1)
		\put(0.5,1){\vector(1,1){0}}
		\qbezier(-0.05,1.35)(2,1.35)(2,-1.5)
		\put(1.9,-0.5){\vector(0,-1){0}}
		\qbezier(-0.05,1.35)(1,1.35)(1,-2.291287847)
		\put(0.9,-0.5){\vector(0,-1){0}}
		\qbezier(-0.05,1.35)(0.2,1.35)(0.2,-2.491987159)
		\put(0.2,-0.5){\vector(0,-1){0}}
		\qbezier(-0.05,1.35)(2.5,1.35)(2.5,0)
		\put(1.9,0.95){\vector(1,-1){0}}
		\put(-0.65,-0.7){\vector(0,-1){0}}
		
		\qbezier(-1.5,0)(-1.2,1.15)(-0.05,1.35)
		\qbezier(-1.5,-0.5)(-1.2,0.55)(-0.4,0.65)
		\qbezier(-0.4,0.65)(0.3,0.45)(0.4,-0.4)
		\put(0.3,0){\vector(0.5,-1){0}}

		\begingroup
		\tiny
		\put(-0.1,0.7){$\theta$}
		\put(-0.5,1.7){$r=0$}
		\put(1,-0.5){$\tau^+_0$}
		\put(2,-0.5){$\tau^+_1$}
		\put(1.9,0.5){$\tau^+_2$}
		\put(-1.7,0.5){$\tau^+_0$}
		\put(-1.7,-0.8){$\gamma^+_0$}
		\put(-1,-1){$\mu^+_0$}
		
		\endgroup
		\end{picture}
		\caption{Meridians on the Zoll sphere.}\label{fig4a}
	\end{center}
\end{figure}
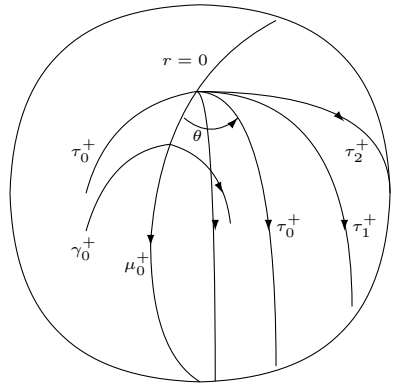

\begin{figure}[h]
	\begin{center}
		
		\setlength{\unitlength}{1cm}
		\begin{picture}(0,5)(0,-2.5)
		\qbezier(-2.5,0)(-2.465,1.25)(-1.85,1.85)
		\qbezier(0,2.5)(-1.25,2.465)(-1.85,1.85)
		\qbezier(2.5,0)(2.465,1.25)(1.85,1.85)
		\qbezier(0,2.5)(1.25,2.465)(1.85,1.85)
		\qbezier(-2.5,0)(-2.465,-1.25)(-1.85,-1.85)
		\qbezier(0,-2.5)(-1.25,-2.465)(-1.85,-1.85)
		\qbezier(2.5,0)(2.465,-1.25)(1.85,-1.85)
		\qbezier(0,-2.5)(1.25,-2.465)(1.85,-1.85)

		\qbezier(-2.5,0)(0,-1.2)(2.5,0)
		\qbezier(0,2.5)(-1,0)(0,-2.5)
		\qbezier[65](-2.5,0)(0,1.4)(2.5,0)
		\qbezier[45](-0.55,-0.555)(-0.55,-0.555)(1.1,0.6)
		\qbezier[45](-1.4,-0.4)(-1.4,-0.4)(1.9,0.3)
		
		\begingroup
		\scriptsize
		\put(0,2.8){$N_M:R=\frac{\pi}{2}$}
		\put(0,-2.8){$S_M:R=-\frac{\pi}{2}$}
		\put(2.7,0){$\{R=0\}$}\
		\put(1,0.5){$\bullet$}
		\put(1.8,0.2){$\bullet$}
		\put(-2.3,-0.2){$\bullet$}
		\put(-2,-0.3){$\bullet$}
		\put(-1.5,-0.5){$\bullet$}
		\put(-0.6,-0.65){$\bullet$}
		\tiny
		\put(1,0.7){$[\mu^-_0]$}
		\put(-1,-1){$[\mu^+_0]$}
		\put(1.8,0.4){$[\tau^-_0]$}
		\put(-1.8,-0.7){$[\tau^+_0]$}
		\put(-2.4,-0.5){$[\tau^+_1]$}
		\put(-2.2,0){$[\tau^+_2]$}
		\endgroup
		
		\end{picture}
		\caption{Equator points on the manifold of geodesics.}\label{fig4b}
	\end{center}
\end{figure}
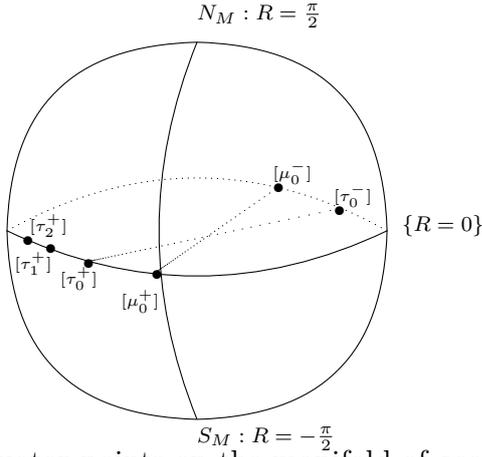


\section{The Finsler metric on the manifold of geodesics}\label{sec:the Finsler metric}

\subsection{The Finslerian indicatrix}\label{subsec:the indicatrix}

	We will express the Jacobi vector fields $Y_1$, $Y_2$ in the canonical coordinates of $TM$, that is
	\begin{proposition}
		We have
		\begin{equation}\label{Y1 Y2 by R, Theta}
		\begin{split}
		Y_1 & =c_1\frac{\partial}{\partial \Theta}\\
		Y_2 & =\frac{1}{c_1\cos r_c}\frac{\partial}{\partial R},
		\end{split}
		\end{equation}
		where $c_1:=\frac{1+h(\cos r_c)}{\cos r_c}$, $r_c=\arcsin c$, $0<c<\frac{\pi}{2}$.
		Moreover 
		\begin{equation}\label{comput y_1', y_2'}
		\begin{split}
		y'_1 & = c_1\frac{\cos r}{1+h(\cos r)}\\
		y'_2 & =\pm\frac{1}{c_1}\frac{\cos r}{1+h(\cos r)}
		\int_{r_c}^r \frac{\sin s}{\cos^2s}\Bigl[1- 
		\frac{\cos s\cdot h'(\cos s)}{1+h(\cos s)} \Bigr]
		\Bigl[\frac{1+h(\cos s)}{\sqrt{\sin^2s-c^2}} \Bigr] ds,
		\end{split}
		\end{equation}
		where $R\leq r \leq \pi-R$, $0<R<\frac{\pi}{2}$, and the sign $\pm$ is the sign of $\dot{r}(t)$.
	\end{proposition}
	
	\begin{proof}
Recall that by definition $Y_1(t)=c_1Y(t)$, where $Y(t)$ is the normal component of $\frac{\partial}{\partial \theta}|_{\gamma(t)}$, hence by the definition of our local coordinates on $M$, first formula in \eqref{Y1 Y2 by R, Theta} follows immediately, by identifying $Y(t)$ with $\frac{\partial}{\partial \Theta}$.

Likewise, if we observe that the unit length vector field in the $\frac{\partial}{\partial r}|_{\gamma(0)}$ direction is 
$$
\frac{1}{\Vert\frac{\partial}{\partial r}|_{\gamma(0)}\Vert
	}
	\frac{\partial}{\partial r}|_{\gamma(0)}=
	\frac{1}{1+h(\cos r_c)}\frac{\partial}{\partial r}|_{\gamma(0)}
$$ 
we obtain the second formula in \eqref{Y1 Y2 by R, Theta}.

For the second set of formulas, observe that by definition $y_1(t)=c_1y(t)$, that is $y_1'(t)=c_1y'(t)$ and using \eqref{comput y'} the first formula in \eqref{comput y_1', y_2'} follows immediately.

Next, we start by taking the derivative of $y_2(t)$ from \eqref{y_1,y_2 formulas}, that is
$$
y_2'(t)=\frac{d}{dt}[\frac{1}{y_1'(t)}]-y_1'(t)\int_0^t
\frac{G(\gamma(s))}{[y'_1(s)]^2}ds-y_1(t)\frac{d}{dt}\int_0^t
\frac{G(\gamma(s))}{[y'_1(s)]^2}ds,
$$
and using Leibnitz chain formula and the Jacobi equation it results
$$
y_2'(t)=-y_1'(t)\int_0^t
\frac{G(\gamma(s))}{[y'_1(s)]^2}ds.
$$
	By using formula for $y_1'(t)$, the expression of the Gauss curvature \eqref{Gauss curv for Zoll surf},  and \eqref{ds over dr}, the desired formula follows immediately after changing the variable from curve parameter $t$ to $r$ due to the obvious relation 
	\begin{equation}\label{integral eq 1}
	\int_0^t
	\frac{G(\gamma(s))}{[y'_1(s)]^2}ds=\frac{1}{c_1^2}
	\int_{r_c}^r \frac{\sin s}{\cos^2s}\Bigl[1- 
	\frac{\cos s\cdot h'(\cos s)}{1+h(\cos s)} \Bigr]
	\Bigl[\pm \frac{1+h(\cos s)}{\sqrt{\sin^2s-c^2}} \Bigr] ds
	\end{equation}
		
		$\qedd$
	\end{proof}
	We recall that due to the local coordinates definition, we have $c=\sin R$, $R=r_c$. 
	
	Observe that the embedding $\iota:\Sigma\to TM$ reads now
	$$
	\iota(\hat{\gamma}(t))=-c_1y_2(t)\frac{\partial}{\partial \Theta}+
	\frac{y_1(t)}{1+h(\cos R)}\frac{\partial}{\partial R},
	$$
	that is, if we denote by $(R,\Theta;v_1,v_2)$ the canonical coordinates on $TM$, 
	\begin{equation}\label{gen formula iota in v1, v2}
	\iota(\hat{\gamma}(t))=v_1(r)\frac{\partial}{\partial R}+
	v_2(r)\frac{\partial}{\partial \Theta},
	\end{equation}
	hence the geodesic flow parametric equations read
	$$
	 v_1(r)= \frac{y_1(t)}{1+h(\cos R
	 	)}
	 =\pm\frac{1}{\cos R}\cdot \sqrt{\sin^2r-c^2}
	$$
	and 
	\begin{equation*}
	\begin{split}
	v_2(r)&=-c_1y_2(t)=-c_1\frac{1}{y'_1(t)}+c_1y_1(t)\int_0^t
	\frac{G(\gamma(s))}{[y'_1(s)]^2}ds\\
&	=-\frac{1+h(\cos r)}{\cos r}	
\pm c_1 \frac{1+h(\cos r_c)}{\cos r_c} \sqrt{\sin^2r-c^2}\int_0^t
\frac{G(\gamma(s))}{[y'_1(s)]^2}ds\\
&= -\frac{1+h(\cos r)}{\cos r}	
+\sqrt{\sin^2r-c^2}\int_{r_c}^r \frac{\sin s}{\cos^2s}\Bigl[1- 
\frac{\cos s\cdot h'(\cos s)}{1+h(\cos s)} \Bigr]
\Bigl[\pm \frac{1+h(\cos s)}{\sqrt{\sin^2s-c^2}} \Bigr] ds,
	\end{split}
	\end{equation*}
	where we have used definition formula for $c_1$, \eqref{y_1,y_2 formulas} and \eqref{integral eq 1}.

	Hence, we get 
	\begin{theorem}\label{thm: indicatrix param eqs}
			The parametric equations of the corresponding Finsler metric of constant flag curvature $K=1$ are 
		\begin{equation}\label{param eq indic ver1}
		\begin{split}
		& v_1(r)= \pm\frac{1}{\cos R}\cdot \sqrt{\sin^2r-c^2}\\
		& v_2(r)= -\frac{1+h(\cos r)}{\cos r}	
		+\sqrt{\sin^2r-c^2}\int_{r_c}^r \frac{\sin s}{\cos^2s}\Bigl[1- 
		\frac{\cos s\cdot h'(\cos s)}{1+h(\cos s)} \Bigr]
		\Bigl[ \frac{1+h(\cos s)}{\sqrt{\sin^2s-c^2}} \Bigr] ds.
		\end{split}
		\end{equation} 	
	\end{theorem}

\begin{remark}
\quad Using the equality
\begin{equation*}
\int\frac{\sin s}{\cos^2 s}\frac1{\sqrt{\sin^2-c^2}}\,ds=\frac1{1-c^2}\frac{\sqrt{\sin^2 s-c^2}}{\cos s}+\text{constant}\,,
\end{equation*}
one can rewrite $v_2$ as
\begin{gather*}
v_2=-[1+h(\cos r)]\frac{\cos r}{\cos^2 R}-
\frac{\sin^2r-c^2}{\cos^2 R}\,h'(\cos r)\\
-\frac{\sqrt{\sin^2 r-c^2}}{\cos^2 R}\int_R^r\sin s\sqrt{\sin^2 s - c^2}\,h''(\cos s)\,ds,
\end{gather*}
which eliminates the apparent singularity
at $r=\pi/2$ in the former expression.
	\end{remark}

\begin{theorem}\label{thm: G vs indic curv}
The strong convexity of the Finsler indicatrix \eqref{param eq indic ver1} is equivalent to the curvature condition $G>0$ of the Zoll metric. 
\end{theorem}
\begin{proof}
Indeed, observe that we have 
	\begin{equation}
	\begin{split}
		v_1(r)& = \frac{1}{1+h(\cos R)}\ y_1(t)\\
		\dot{v}_1(r) &=\frac{1}{1+h(\cos R)}\ y_1'(t)\frac{dt}{dr}\\
		\ddot{v}_1(r)&=\frac{1}{1+h(\cos R)}\ 
		\Bigl[
		y_1''(t)\Bigl(\frac{dt}{dr}\Bigr)^2+y_1'(t)\frac{d^2t}{dr^2}
		\Bigr]
	\end{split}
	\end{equation}

	and 
\begin{equation}
\begin{split}
	v_2(r)& =-\frac{1+h(\cos R)}{\cos R}\ y_2(t)\\
		\dot{v}_2(r) &=-\frac{1+h(\cos R)}{\cos R}\ y_2'(t)\frac{dt}{dr}\\
		\ddot{v}_2(r)&=-\frac{1+h(\cos R)}{\cos R}\ 
		\Bigl[
		y_2''(t)\Bigl(\frac{dt}{dr}\Bigr)^2+y_2'(t)\frac{d^2t}{dr^2}
		\Bigr]
\end{split}
\end{equation}
where $\dot{v}_i$ and $\ddot{v}_i$ are the first and second derivative of $v_i(r)$ with respect to $r$, for $i\in\{1,2\}$. 

We can compute now the curvature $k(r)$ of the indicatrix curve 
$(v_1(r),v_2(r))$, and by using the Jacobi equation, we get
\begin{equation}
k(r):=\dfrac{\ddot{v}_1\dot{v}_2-\ddot{v}_2\dot{v}_1}{\dot{v}_1v_2-\dot{v}_2v_1}=\Bigl(\frac{dt}{dr}\Bigr)^2G,
\end{equation}
hence the conclusion follows from our construction.
$\qedd$
\end{proof}


	\subsection{The round sphere case $h=0$}\label{Example: Riemannian case}
	The trivial example of a Zoll sphere of revolution is the canonical Riemannian sphere given by $h=0$. In this case, the indicatrix equations read
	\begin{equation*}
	\begin{split}
	& v_1(r)= \frac{1}{\cos R}\cdot \sqrt{\sin^2r-c^2}\\
	& v_2(r)=  -\frac{1}{\cos r}	+
	\frac{1}{\cos^2 R}
	\frac{{\sin^2 r-c^2}}{\cos r}
	=
	\frac{\sin^2r+\cos^2 R-c^2}{\cos^2R\cos r}=\frac{\sin^2r-1}{\cos^2R\cos r}
	=-\frac{\cos r}{\cos^2 R}.
	\end{split}
	\end{equation*}

	Observe that this is equivalent to 
	\begin{equation*}
	\begin{split}
	\sin^2 r & =v_1^2 \cos^2R+c^2\\
	\cos^2 r & = v_2^2\cos^4 R
	,
	\end{split}
	\end{equation*}
	and by adding these relations we get 
	\begin{equation}\label{ellipse for h=0}
	v_1^2+\cos^2R\cdot v_2^2=1,
	\end{equation}
	where we have used $1-c^2=\cos ^2R$. 
	
	It is clear that this is an ellipse in $T_pM$ corresponding to the Riemannian canonical sphere, therefore we have
	\begin{proposition}
		In the case $h=0$, the corresponding metrical structure on $M$ is the canonical Riemannian sphere. 
	\end{proposition}
	\begin{remark}
		Observe that our local coordinates on $M$ are $(R,\Theta)\in [-\frac{\pi}{2},\frac{\pi}{2} ]\times \R\slash 2\pi\mathcal Z$, and that by the coordinates changing $\widetilde{R}:=R+\frac{\pi}{2}$ we obtain $\widetilde{R}\in [0,\pi]$ and hence \eqref{ellipse for h=0} becomes 	$v_1^2+\sin^2R\cdot v_2^2=1$, that is the expected ellipse in the case of the canonical Riemannian sphere. 
	\end{remark}


	\section{Examples}\label{sec:Examples}
	
\subsection{The case when $h$ is polynomial}\label{subsec: h polynomial}

Let us consider the case when the function $x\mapsto h$ is a polynomial of odd order, that is
\begin{equation}\label{the polynom h}
h(x)=\sum_{k=0}^na_{2k+1}x^{2k+1},
\end{equation}
where $n\geq 1$, i.e. the degree of $h$ is greater than two, provided $G>0$, and $\sum_{k=0}^na_{2k+1}=0$, with $G$ given in \eqref{Gauss curv for Zoll surf}. It can be seen that this function $h$ satisfies the conditions needed to induce a Zoll metric on $\Sph^2$, see Section \ref{sec: Zoll surfaces}.

Observe that 
$$
h'(x)=\sum_{k=0}^n (2k+1)a_{2k+1}x^{2k}
$$
and 
$$
h''(x)=
\sum_{k=1}^n 2k(2k+1)a_{2k+1}x^{2k-1}=\sum_{k=0}^{n-1} b_{2k+1}x^{2k+1},
$$
where we put $b_{2k+1}:=2(k+1)(2k+3)a_{2k+3}$, for all integers $k\in \{0,\dots,n-1\}$.

\begin{lemma}
	For a real number $\lambda\geq x$ and any integer $k\in \{0,\dots,n-1\}$, we have
	$$
	\int \sqrt{\lambda^2-x^2}\ x^{2k+1}dx=
	\lambda^{2k}(\lambda^2-x^2)^{3\slash 2}\sum_{\rho=0}^k\frac{(-1)^{\rho+1}}{(2\rho+3)\lambda^{2\rho}}\binom{k}{\rho}(\lambda^2-x^2)^\rho+\mathcal C
	,
	$$
	where $\mathcal C$ is arbitrary constant.
\end{lemma}

Indeed, if we use the substitution $u:=\lambda^2-x^2$, it follows
$$
x^{2k}=(\lambda^2-u)^k=\sum_{\rho=0}^k(-1)^\rho\binom{k}{\rho}\lambda^{2(k-\rho)}u^\rho,
$$
and hence, we have
\begin{equation}
\begin{split}
&\int \sqrt{\lambda^2-x^2}\ x^{2k+1}dx
=-\frac{1}{2}\int \sqrt{u}(\lambda^2-u)^kdu\\
&=\frac{1}{2}\sum_{\rho=0}^k(-1)^{\rho+1}\binom{k}{\rho}\lambda^{2(k-\rho)}\int \sqrt{u}\ u^\rho du\\
&=\lambda^{2k}\sum_{\rho=0}^k\frac{(-1)^{\rho+1}}{(2\rho+3)\lambda^{2\rho}}\binom{k}{\rho}(\lambda^2-x^2)^{\rho+3\slash 2}
+\mathcal C,
\end{split}
\end{equation}
where $\mathcal C$ is arbitrary constant.

Next, we compute
\begin{equation}
\begin{split}
&\int \sqrt{\lambda^2-x^2}\ h''(x)dx=\sum_{k=0}^{n-1}b_{2k+1}\int \sqrt{\lambda^2-x^2}\ x^{2k+1}dx\\
&=(\lambda^2-x^2)^{3\slash 2}\sum_{k=0}^{n-1}b_{2k+1}\lambda^{2k}\sum_{\rho=0}^k\frac{(-1)^{\rho+1}}{(2\rho+3)\lambda^{2\rho}}\binom{k}{\rho}(\lambda^2-x^2)^\rho+\mathcal C
,
\end{split}
\end{equation}
where $\mathcal C$ is an arbitrary constant.

Moreover, we have
\begin{equation}
\begin{split}
&\int_R^r \sin s\sqrt{\sin^2s-c^2}\ h''(\cos s)ds
=-\int_{\cos R}^{\cos r}\sqrt{\cos^2R-x^2}\ h''(x)dx,   \\
&\end{split}
\end{equation}
where we have used $c=\sin R$ and the substitution $x=\cos s$. 

It follows 
\begin{equation}
\begin{split}
&\int_R^r \sin s\sqrt{\sin^2s-c^2}\ h''(\cos s)ds\\
&=   (\cos^2R-\cos^2r)^{3\slash 2}\sum_{k=0}^{n-1}b_{2k+1}\cos^{2k}R\sum_{\rho=0}^k\frac{(-1)^{\rho}}{(2\rho+3)\cos^{2\rho}R}\binom{k}{\rho}(\cos^2R-\cos^2r)^\rho\\
&= \cos^3R\ v_1^3 \sum_{k=0}^{n-1}c_{2k+1}
\sum_{\rho=0}^kd_\rho v_1^{2\rho}
\end{split}
\end{equation}
where we use $\sin^2r-c^2=\cos^2R-\cos^2r=\cos^2 R\ v_1^2$, and put $c_{2k+1}:=b_{2k+1}\cos^{2k}R$, and 
$d_\rho:=\frac{(-1)^{\rho}}{(2\rho+3)}\binom{k}{\rho}$.

On the other hand, observe that
\begin{equation}
\begin{split}
\frac{1+h(x)}{\cos^2R}x\frac{\cos^2R-x^2}{\cos^2R}h'(x)&=\frac{x}{\cos^2R}+\frac{1}{\cos^2R}\sum_{k=0}^na_{2k+1}\Bigl[
(2k+1)\cos^2R-2kx^2
 \Bigr]x^{2k}\\
 &=\pm\frac{\sqrt{1-v_1^2}}{\cos R}+\sum_{k=0}^na_{2k+1}\cos^{2k}R(1+2kv_1^2)(1-v_1^2)^k,
\end{split}
\end{equation}
where we have used $x:=\cos r$ and hence $x^2=\cos^2R(1-v_1^2)$. 

We obtain
\begin{theorem}\label{thm: implicit equation for h(x) polyn}
	In the case when $h$ is the odd polynomial of degree $2n+1$ given in \eqref{the polynom h} satisfying $G>0$, then the implicit equation of the Finsler indicatrix is
	\begin{equation}
	\begin{split}
		v_2&=\pm\frac{\sqrt{1-v_1^2}}{\cos R}-\sum_{k=0}^na_{2k+1}\cos^{2k}R(1+2kv_1^2)(1-v_1^2)^k\\
		&\pm\cos^2Rv_1^4\sum_{k=0}^{n-1}b_{2k+1}\cos^{2k}R
		\sum_{\rho=0}^k\frac{(-1)^{\rho}}{(2\rho+3)}\binom{k}{\rho} v_1^{2\rho},
	\end{split}
	\end{equation}
	where we put $b_{2k+1}:=2(k+1)(2k+3)a_{2k+3}$, for all integers $k\in \{0,\dots,n-1\}$.
\end{theorem}

We remark that the sums in the right hand side looks like $2n+2$-polynomials, but it is easy to see that actually they are only $2n$-degree polynomials. We give here only the general form.


	

	\subsection{Example 1. The case $h(x)=\ve(1-x^2)x$}\label{subsec: Eg. 1}
	
	We will consider the case 
	$$
	h:[-1,1]\to (-1,1),\quad h(x)=\ve(1-x^2)x^{2n+1},
	$$
	$n$ nonnegative integer, $\ve$ small positive constant
	that obviously satisfies all conditions needeed, including positiv sectional curvature. 
	
	For the sake of simplicity, we will consider the case 
\begin{equation}
	h:[-1,1]\to (-1,1),\quad h(x)=\ve(1-x^2)x,
\end{equation}	
with $0<\ve<\frac{1}{2}$. Observe that for this $h$, we have
$$
h'(x)=\ve(1-3x^2),\quad h''(x)=-6\ve x,
$$
and the Gauss curvature 
\begin{equation}\label{G in eg1}
G(x)=-\frac{2\ve x^3+1}{(\ve x^3-\ve x-1)^3}
\end{equation}
 of $(\Lambda,g)$ is a smooth function, with only one real zero at $x=-\frac{\sqrt[3]{4\ve^2}}{2\ve}$ taking the value 1 at $x=0$, and $G(-1)=1-2\ve$, 
that is $G>0$ for any $x\in[-1,1]$ and any $\ve<\frac{1}{2}$.

	In this case from Theorem \ref{thm: implicit equation for h(x) polyn}
	we have
	\begin{equation}\label{ex1:implicit formula}
	\frac{1-v_1^2}{\cos^2R}=(-v_2+\ve v_1^2\cos^2R  -\ve c^2)^2.
	\end{equation}
	
		For obtaining the algebraic equation in $F$, the fundamental function of this Finsler space, we simply substitute $v_i$ by $\frac{v_i}{F}$, $i=1,2$ since this is the indicatrix equation  where $F=1$. It follows 
		$$
		\frac{F^2-v_1^2}{\cos^2 R}=\frac{(\ve c{^2F^2-\ve v_1^2\cos^2R+Fv_2})^2}{F^2},
		$$
		and from here we obtain the 4-th order equation in $F$, whose solution gives the explicit form of the desired Finsler metric. Since the formulas are quite complicated, instead of writing them explicitely here (this can be done very easily) with the substitution $c=\sin R$, the implicit equation \eqref{ex1:implicit formula} can be used to obtain some graphical representations of the indicatrices for different values of the parameter $\ve$ at different points of $M$.
		
		The invariants $I$ and $J$ of this Finsler surface can be easily be obtained using \eqref{G in eg1} and \eqref{I, J from G & theta}.

\begin{figure}[H]
	\centering
	\begin{subfigure}[b]{0.3\textwidth}
		\includegraphics[width=\textwidth]{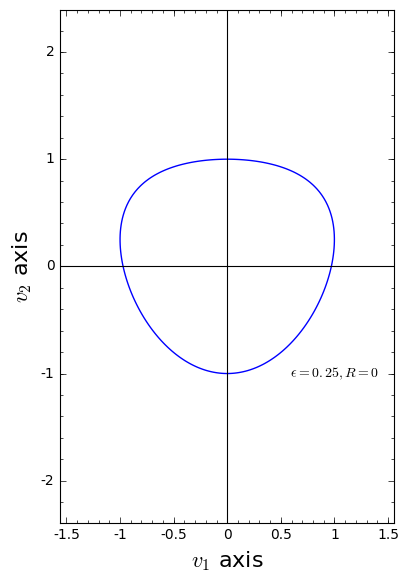}
	\end{subfigure}\quad
	\begin{subfigure}[b]{0.3\textwidth}
	\includegraphics[width=\textwidth]{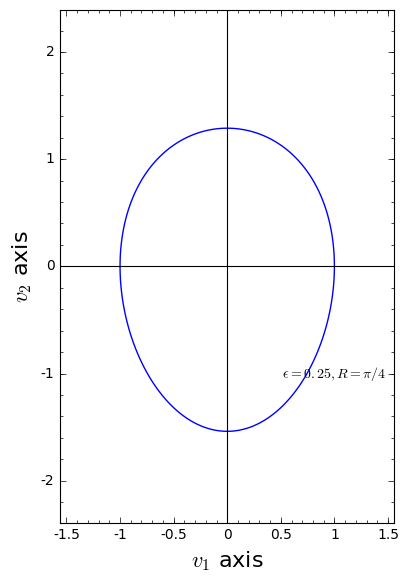}
	\end{subfigure} \\
	\begin{subfigure}[b]{0.3\textwidth}
	\includegraphics[width=\textwidth]{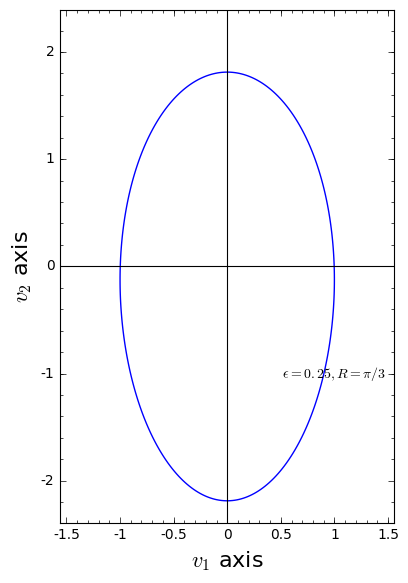}
	\end{subfigure}\quad
	\begin{subfigure}[b]{0.3\textwidth}
		\includegraphics[width=\textwidth]{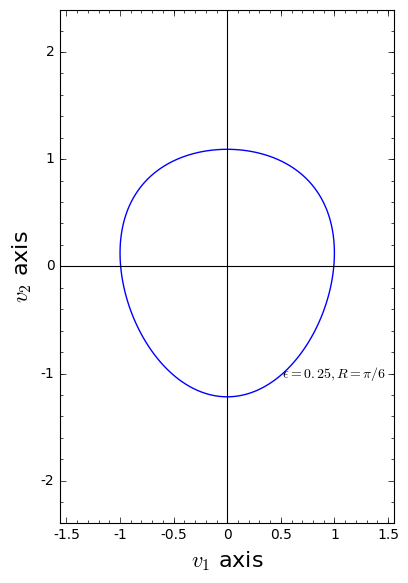}
	\end{subfigure}\quad
	\caption{Indicatrices in Example 1 for $\ve=0.25$.}
\end{figure}
\begin{figure}[H]
	\centering
	\begin{subfigure}[b]{0.3\textwidth}
		\includegraphics[width=\textwidth]{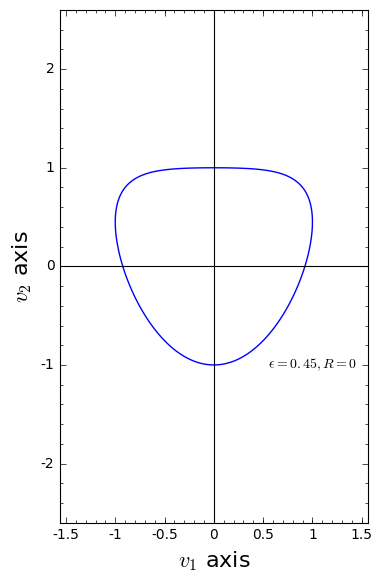}
	\end{subfigure}\quad
	\begin{subfigure}[b]{0.3\textwidth}
		\includegraphics[width=\textwidth]{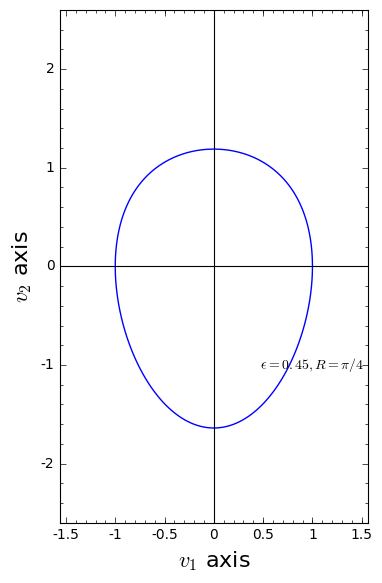}
	\end{subfigure} \\
	\begin{subfigure}[b]{0.3\textwidth}
		\includegraphics[width=\textwidth]{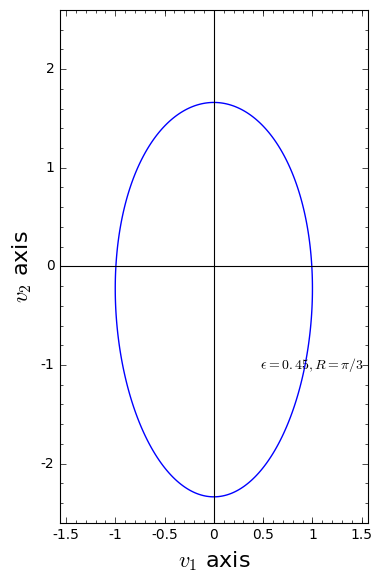}
	\end{subfigure}\quad
	\begin{subfigure}[b]{0.3\textwidth}
		\includegraphics[width=\textwidth]{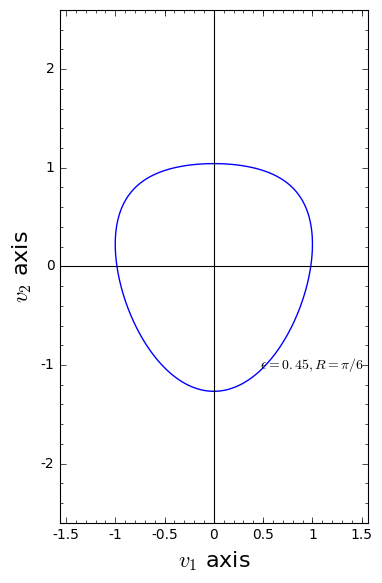}
	\end{subfigure}\quad
	\caption{Indicatrices  in Example 1 for $\ve=0.45$.}
	\label{fig: indicatrices eg 1}
\end{figure}


\subsection{Example 2. The case $h(x)=x(1-x^2)^2$}\label{subsec: Eg. 2}

Another possible choice is the case 
$$
h:[-1,1]\to (-1,1),\quad h(x)=x(1-x^2)^n,
$$
$n>1$ integer 
that obviously satisfies all conditions needeed, including positiv sectional curvature. For the sake of simplicity, we will consider the case 
\begin{equation}
h:[-1,1]\to (-1,1),\quad h(x)=x(1-x^2)^2.
\end{equation}	

Observe that for this $h$, we have
$$
h'(x)=1-6x^2+5x^4,\quad h''(x)=-12x+20x^3,
$$
and the Gauss curvature
\begin{equation}\label{G in eg2}
G(x)=\frac{1+h(x)-xh'(x)}{[1+h(x)]^3}=
\frac{1+4x^3-4x^5}{(1+x-2x^3+x^5)^3}.
\end{equation}

It can be easily seen that $G(x)>0$ for $x\in [-1,1]$. Indeed, observe that $G(-1)=G(1)=1$, and that the critical points and the critical values of the function $G:[-1,1]\to \R$ are 
\begin{equation*}
\begin{split}
&x_0=0.33,\quad G(x_0)=0.56\\
&x_1=0.88,\quad G(x_1)=1.42\\
&x_2=-0.35,\quad G(x_2)=2.18\\
&x_3=-0.81,\quad G(x_3)=0.36,
\end{split}
\end{equation*}
and since $G$ is continuous function it follows that cannot take nonpositive values. 

In this case from Theorem \ref{thm: implicit equation for h(x) polyn}
we have
\begin{equation}
\frac{1-v_1^2}{\cos^2R}=\frac{1}{9} \, {\Bigl[v_{1}^{4} \cos^{4}R + 6 \, v_{1}^{2} \cos^{2}R \ \sin^{2}R - 3 \, \sin^{4}R - 3 \, v_{2}\Bigr]}^{2} 
\end{equation}
that leads to a 8-th order polynomial in $F$. 

Again, the invariants $I$ and $J$ of this Finsler surface can be easily be obtained using \eqref{G in eg2} and \eqref{I, J from G & theta}.



\begin{thebibliography}{[BCS2000]}
	
	
	
	\bibitem{BCS2000}
	D. {Bao}, S.S. {Chern}, Z. {Shen},
	{An Introduction to Riemann--Finsler Geometry}, Springer, GTM 200,
	2000.
	
	
	\bibitem{Br1995}
	{Bryant},~R.,
	{\it Finsler structures on the 2-sphere satisfying $K=1$}, Finsler Geometry,
	Contemporary Mathematics {\bf 196} (1996), 27--41.
	
	
	\bibitem{Br2002}
	{Bryant},~R.,
	{\it Some remarks on Finsler manifolds with constant flag curvature},
	Houston Journal of Mathematics, {\bf vol. 28, no.2} (2002), 221--262.  
	
      \bibitem{Be}
        {Besse}, A., Manifolds all of whose geodesics are closed, Springer-Verlag, 1978.
        
        \bibitem{K1}
        Kiyohara, K., {Compact Liouville Surfaces}, 
        J. Math. Soc. Japan,
        Volume 43, Number 3 (1991), 555-591.
        
        \bibitem{LBM}
        LeBrun, C., Mason, J. L., {\it Zoll manifolds and complex surfaces}, J. Diff. Geometry, {\bf 61} (2002), 453--535.
        
         \bibitem{MS1}
         Matveev, V., Shevchishin, V., {\it Differential invariants for cubic integrals of geodesic flows on surfaces}, J. Geom. Phys. 60(2010) no. 6-8, 833-856.
        
        
 \bibitem{SSS2012}   Sabau, S. V., Shibuya, K., Shimada, H., {\it Moving frames on generalized Finsler structures}, J. Korean Math. Soc. 49 (2012), no. 6, 1229–1257. 
        
         \bibitem{SSP2014}   Sabau, S. V., Shibuya, K., Pitis, Gh., {\it Generalized Finsler structures on closed 3-manifolds}, Tohoku Math. J., vol. 66, no. 3 (2014), 321--353. 
        
\end{thebibliography}
\end{document}